\documentclass{amsart}
\usepackage[utf8]{inputenc}

\usepackage{amsmath,amsfonts, amssymb}
\usepackage{graphicx}
\usepackage{xcolor}
\usepackage{enumerate}
\usepackage{enumitem}
\usepackage{cite}
\usepackage{tabu}
\usepackage{mathtools}
\usepackage[e]{esvect}
\usepackage{booktabs}

\usepackage{stmaryrd} 

\usepackage{fullpage}
\usepackage[active]{srcltx}

\usepackage{changepage}
\usepackage{chngcntr}

\usepackage{dbnsymb}
\usepackage{multibib}

\usepackage{url}
\usepackage{hyperref}
\hypersetup{
	colorlinks,
	linkcolor={red!50!black},
	citecolor={blue!50!black},
	urlcolor={blue!80!black}
}  
\usepackage{amsmath}

\newcounter{relctr} 

\newcommand\labelrel[2]{%
  \begingroup
    \refstepcounter{relctr}%
    \stackrel{\textnormal{(\alph{relctr})}}{\mathstrut{#1}}%
    \originallabel{#2}%
  \endgroup
}
\AtBeginDocument{\let\originallabel\label} 

\theoremstyle{plain}
\newtheorem{theorem}{Theorem}[section]

\newtheorem{cor}[theorem]{Corollary}

\newtheorem{lemma}[theorem]{Lemma}

\newtheorem{question}[theorem]{Question}

\newtheorem*{thm:main}{Theorem~\ref{thm:main}}
\newtheorem*{thm:GammaSig}{Theorem~\ref{thm:GammaSig}}
\newtheorem*{thm:chi-orb s-m}{Theorem~\ref{chi-orb s-m}}

\theoremstyle{definition}
\newtheorem{defn}[theorem]{Definition}
\newtheorem{remark}[theorem]{Remark}
\newtheorem{example}[theorem]{Example}

\newcommand{\comment}[1]{}

\newcommand{\bdry}{\ensuremath{\partial}}

\newcommand{\boundary}{\bdry}

\newcommand{\nbhd}{\ensuremath{\mathcal{N}}}

\newcommand{\N}{\ensuremath{\mathbb{N}}}

\newcommand{\R}{\ensuremath{\mathbb{R}}}
\newcommand{\Z}{\ensuremath{\mathbb{Z}}}

\newcommand{\RP}{\ensuremath{\mathbb{RP}}}

\newcommand{\sv}{\ensuremath{\mathrm{sv}}}

\newcommand{\cut}{\ensuremath{\backslash}}

\newcommand{\thetan}{\theta}
\newcommand{\Klein}{\scalebox{.8}{\rotatebox[origin=c]{180}{\ensuremath{\YGraph}}}}

\newcommand{\mir}[1]{m #1}

\definecolor{amaranth}{rgb}{0.9, 0.17, 0.31} 
\definecolor{carrotorange}{rgb}{0.93, 0.57, 0.13} 
\definecolor{citrine}{rgb}{0.89, 0.82, 0.04} 
\definecolor{dartmouthgreen}{rgb}{0.05, 0.5, 0.06} 
\definecolor{ballblue}{rgb}{0.13, 0.67, 0.8} 
\definecolor{ceruleanblue}{rgb}{0.16, 0.32, 0.75} 
\definecolor{amethyst}{rgb}{0.6, 0.4, 0.8} 
\definecolor{amber}{rgb}{1.0, 0.75, 0.0} 
\definecolor{burlywood}{rgb}{0.87, 0.72, 0.53} 

\title{Signature, slicing foams, and crossing changes of Klein graphs} 

\makeatletter
\@namedef{subjclassname@1991}{2020 Mathematics Subject Classification}
\makeatother

\author{Kenneth L. Baker}
\address{Department of Mathematics, 
University of Miami}   
\email{k.baker@math.miami.edu}     

\author{Allison H. Moore}
\address{Department of Mathematics \& Applied Mathematics,  Virginia Commonwealth University}   
\email{moorea14@vcu.edu}   

\author{Danielle O'Donnol}
\address{School of Science, Mathematics, and Engineering,
Marymount University}   
\email{dodonnol@marymount.edu} 

\author{Scott Taylor}
\address{Department of Mathematics, 
Colby College}   
\email{scott.taylor@colby.edu}



\date{\today}
\begin{document}

\begin{abstract}
    A totally oriented Klein graph is a trivalent spatial graph in the 3-sphere with a 3-coloring of its edges and an orientation on each bicolored link. A totally oriented Klein foam is a 3-colored 2-complex in the 4-ball whose boundary is a Klein foam and whose bicolored surfaces are oriented. We extend Gille-Robert's signature for 3-Hamiltonian Klein graphs to all totally oriented Klein graphs and develop an analogy of Murasugi's bounds relating the signature, slice genus and unknotting number of knots. In particular, we show that the signature of a totally oriented Klein graph produces a lower bound on the negative orbifold Euler characteristic of certain totally oriented Klein foams bounded by $\Gamma$. When $\Gamma$ is abstractly planar, these negative Euler characteristics, in turn, produce a lower bound on a certain natural unknotting number for $\Gamma$. \emph{Mutatis mutandi}, we produce lower bounds on the corresponding Gordian distance between two totally oriented Klein graphs that can be related by a sequence of crossing changes. We also give examples of $\thetan$-curves for which our lower bounds on unknotting number improve on previously known bounds.  
\end{abstract}

\maketitle

\section{Introduction}

For a knot $K \subset S^3$, the familiar chain of inequalities
\[ \frac12 |\sigma(K)| \leq g_4(K) \leq u(K)\]
between the signature 
$\sigma(K)$, slice genus $g_4(K)$, and unknotting number $u(K)$ \cite[Theorem 9.1 and Theorem 10.1]{Murasugi} can be rewritten as
\[
|\sigma(K) | \leq 1 - \chi_4(K) \leq 2u(K).
\]
Here  $\chi_4(K)$ is the \emph{slice Euler characteristic}, the maximal Euler characteristic of a properly embedded, orientable smooth surface $F \subset B^4$ without sphere components bounded by $K$.
Considering the knot $K$ as the branch locus for a branched double cover of $S^3$, and the surface $F$ as a branch locus for a branched double cover of $B^4$, we can consider the orbifold Euler characteristic $\chi^{orb}(F) =\dfrac12 \chi(F)$ so that $\chi^{orb}_4(K)$ is the maximum of $\chi^{orb}(F)$ among such surfaces $F$. The division by $2$ in the orbifold Euler characteristic is due to the degree of the cover.  Now our initial inequalities takes the form: 
\begin{equation}
\label{eqn:sig-chi-unknot_for_knots}
\frac{1}{2}|\sigma(K) | \leq \frac{1}{2} - \chi^{orb}_4(K) \leq u(K).    
\end{equation}

The \emph{Gordian distance} $d(K_1, K_2)$ between two knots is the minimal number of crossing changes needed to convert $K_1$ into $K_2$. Essentially the same proof as for Inequality \eqref{eqn:sig-chi-unknot_for_knots} shows that
\begin{equation}
\label{eqn:sig-chi-gorddist_for_knots}
\frac{1}{2}|\sigma(K_1) - \sigma(K_2)| \leq \frac{1}{2} - \chi^{orb}_4(K_1 \# -mK_2) \leq d(K_1, K_2). 
\end{equation}
Indeed, $u(K)$ is just the Gordian distance of $K$ with the unknot.

Both inequalities (\ref{eqn:sig-chi-unknot_for_knots}) and (\ref{eqn:sig-chi-gorddist_for_knots}) extend to the setting of totally oriented Klein graphs, which are trivalent spatial graphs with extra coloring and orientation information.  The simplest Klein graph is the $\thetan$-curve with edge colorings.  In a sequence of crossing changes of such a graph it is sensible to keep track of whether a crossing change occurs between distinct edges, a mixed-color crossing change, or between an edge and itself, a same-color crossing change. For a sequence of crossing changes, we let $m$ denote the number of mixed-color crossing changes and $s$ the number of same-color crossing changes. The \emph{Klein unknotting number} $u_{\Klein}(\thetan)$ is the minimum of $m/2 + s$ over all crossing change sequences converting $\thetan$ to the unknot. We prove:

\begin{cor}\label{cor:sigsliceunknot}
If $\thetan$ is a theta curve in $S^3$, then
\[
\frac14 |\sigma(\thetan)| \leq \frac14 - \chi^{orb}_4(\thetan) \leq u_{\Klein}(\thetan).
\]
\end{cor}
The statement for theta curves is a corollary of Theorem \ref{thm:main}.

Recently, the quest to understand unknotting numbers and sequences of crossing changes for $\thetan$-curves has been provoked by potential biological applications \cite{BOD, BOS, BBMOT}. In Section \ref{sec:examples}, we give examples demonstrating that the bounds of Corollary \ref{cor:sigsliceunknot} improve previous bounds.

\subsection{Totally orientable Klein graphs}
More generally, 
we develop analogous bounds for the class of \emph{totally oriented Klein graphs}.

A {\em spatial graph} is a smooth embedding in $S^3$ of a compact graph, allowing $S^1$ components, considered up to smooth isotopy.  A trivalent spatial graph is a {\em web}.  Without the embedding into $S^3$, a web is also known as a \emph{cubic graph}.  A {\em Klein coloring} or {\em Tait coloring} of a web is an assignment of a color $r,g,b$ to each edge or loop so that edges of all three colors are incident to each vertex of the web.   A web equipped with a Klein coloring is a {\em Klein graph}. Klein colorings arise naturally when considering quotients of the Klein group acting on 3-manifolds.

Given a Klein graph $\Gamma$, each unordered pair $\{i,j\} \subset \{r,g,b\}$ defines the {\em bicolored link} $\Gamma_{ij}$ comprised of only the edges and loops colored $i$ or $j$. A Klein graph is \emph{3-Hamiltonian} if each bicolored link is actually a knot; signatures for 3-Hamiltonian Klein graphs were developed by Gille and Robert \cite{GR}. In the process of extending their work to relate signature, unknotting number, and $\chi_4^{orb}$, we found that it was possible to drop the 3-Hamiltonian condition. 

This enables an analogue to the classical extension of Inequality \eqref{eqn:sig-chi-unknot_for_knots} to links. Effectively we obtain that for two Klein graphs related by crossing changes, the difference in their signatures gives a lower bound on an Euler characteristic of a certain type of cobordism between them. This, in turn, gives a lower bound on the Klein Gordian distance between them, though with adjustments based on easily calculable invariants of the two graphs. 

\begin{thm:main}
Suppose that $\vv{\Gamma_1}$ and $\vv{\Gamma_2}$ are totally oriented Klein graphs that are related by a sequence of crossing changes. 
Set $V:=|V(\Gamma_1)|=|V(\Gamma_2)|$ and $\mu:= \mu(\Gamma_1)=\mu(\Gamma_2)$. Then:

\begin{align*}
 |\sigma(\vv{\Gamma_1})-\sigma(\vv{{\Gamma_2}})| - \beta(\vv{\Gamma_1})-\beta(\vv{{\Gamma_2}}) +4\mu  - 12 
 &\leq 
 5- 4 \chi^{orb}_4(\vv{\Gamma_1} \#_3 -\mir{\vv{\Gamma_2}}) -2V \\
 &\leq 
 - 4 \chi^{orb}_4(\vv{\Gamma_1}, \vv{\Gamma_2}; s) -2V\\
 &\leq 
 4d_{\Klein}(\vv{\Gamma_1}, \vv{\Gamma_2})    
\end{align*}
where $\vv{\Gamma_1} \#_3 -\mir{\vv{\Gamma_2}}$ is any vertex connected sum of $\vv{\Gamma_1}$ and $-\mir{\vv{\Gamma_2}}$ compatible with the total orientations.
\end{thm:main}

Here $\beta$ is the nullity, $\mu$ is the sum of the number of components of the bicolored links, and $d_{\Klein}$ is the Klein Gordian distance (akin to the Klein unknotting number). To obtain Theorem \ref{thm:main} we must impose the extra structure of a \emph{total orientation} on our Klein graphs, which is a choice of orientation on each bicolored link.  We denote a totally oriented Klein graph by $\vv{\Gamma}$.

The additional structure of a total orientation enables us to define in Section \ref{sec:sigs} for any totally oriented Klein graph $\vv{\Gamma}$ signatures $\zeta(\Gamma)$ and $\sigma(\vv{\Gamma})$. These are analogous to the Murasugi signature of an unoriented link and the signature of an oriented link, respectively, and extending \cite{GR}. 
The signature $\sigma(\vv{\Gamma})$ derives from a $\Z_2 \times \Z_2$  covering of $B^4$ branched over a properly embedded colored $2$-complex $F$ in $B^4$ called a {\em Klein foam} for which $\Gamma = \bdry F$ in  $S^3 = \bdry B^4$. (See Definition~\ref{defn:foam} for the specifics of the foam.) The terms $\chi^{orb}_4(\vv{\Gamma_1} \#_3 -\mir{\vv{\Gamma_2}})$ and  $\chi^{orb}_4(\vv{\Gamma_1}, \vv{\Gamma_2}; s)$ appearing in Theorem \ref{thm:main} both refer to \emph{slice orbifold characteristics} of two types of totally oriented Klein foams relating $\vv{\Gamma_1}$ and $\vv{\Gamma_2}$.  Here, $\chi^{orb}_4(\vv{\Gamma_1} \#_3 -\mir{\vv{\Gamma_2}})$ is a maximum taken over all of orientable slice foams bounding the vertex sum $\vv{\Gamma_1} \#_3 -\mir{\vv{\Gamma_2}}$, whereas $\chi^{orb}_4(\vv{\Gamma_1}, \vv{\Gamma_2}; s)$ calculates the \emph{seamed cobordism characteristic} of a restricted type of foam 
where the singular set is a 1-manifold of which each edge runs from $\vv{\Gamma_1}$ to $\vv{\Gamma_2}$
(see Definitions \ref{defn:foamycobordism} and \ref{defn:chi-orb}).

In general the singular set of a Klein foam could have interior vertices.  A total orientation on a Klein foam bestows an orientation on these vertices. A signed count of these vertices, the \emph{signed seam vertex count} $\operatorname{sv}$, can be determined from the totally oriented Klein graph that is the boundary of the foam.   This is an invariant of totally oriented Klein graphs, as we discuss in Section~\ref{sec:totalorientationsverticesandseamvertices}.

\begin{thm:GammaSig}
For a totally oriented Klein graph $\vv{\Gamma} \subset S^3$, we have 
\[ |\sigma(\vv{\Gamma})| \leq 3-|V(\Gamma)|+2|\sv(\vv{\Gamma})|-4\chi^{orb}_4(\vv{\Gamma}) -2(\mu(\Gamma)-3) + \beta(\vv{\Gamma}).\]
\end{thm:GammaSig}

In Section \ref{sec:cd}, we turn our attention to the construction of the totally oriented seamed foamy cobordisms that are required for the third inequality in Theorem \ref{thm:main}. As it will turn out, such objects naturally arise by stacking cobordisms resulting from crossing changes of totally oriented Klein graphs. 

\begin{thm:chi-orb s-m} 
 Suppose there is a sequence of $s$ same-colored crossing changes and $m$ mixed-colored crossing changes between totally oriented Klein graphs $\vv{\Gamma_1}$ and $\vv{\Gamma_2}$.   Then there is a totally orientable seamed foamy cobordism $\vv{F}$ between $\vv{\Gamma_1}$ and $\vv{\Gamma_2}$ without bicolored spheres with 
    \[ \chi^{orb}(\vv{F}) = -|V(\vv{\Gamma_1})|/2 - (s+m/2).\]
\end{thm:chi-orb s-m}

Consequently, the seamed cobordism characteristic provides a lower bound on Klein Gordian distance, providing the last inequality in Theorem \ref{thm:main}:
\begin{cor}
\label{chi-orb dist}
If the Klein Gordian distance between two Klein graphs $\Gamma_1$ and $\Gamma_2$ is $d_{\Klein}(\Gamma_1, \Gamma_2)$ then 
        \[ -\chi_4^{orb}(\Gamma_1, \Gamma_2;s) \leq   d_{\Klein}(\Gamma_1, \Gamma_2)+\frac12 |V(\Gamma_1)|. 
        \]
\end{cor}

\subsection{Organization}

Section \ref{sec:defs} collects all major definitions, including Klein graphs and unknotting numbers in \ref{subsec:klein}, orientations and foams in \ref{sec:littleorientations}, orbifold Euler characteristics in \ref{subsec:orbifolds}, and signatures in \ref{sec:sigs}. In Section \ref{sec:totalorientations}, we study the structure of total orientatons on Klein graphs, addressing foams with seam vertices in \ref{subsec:foams with seams}, orientations in \ref{sec:totalorientationsverticesandseamvertices}, and cobordisms in \ref{subsec:cobordisms}. 
In Section \ref{sec:results} we prove our main results: Theorem \ref{thm:main}, Theorem \ref{thm:GammaSig} and Theorem \ref{chi-orb s-m}.
In Section \ref{sec:examples} we present several examples; in \ref{subsec:improvebounds} we give improved bounds on unknotting number for a family of theta curves, and in \ref{subsec:strong sig} we compare the bounds of Theorem \ref{thm:main} with the notion of \emph{strong} signature. Lastly, an appendix aggregates the behavior of various invariants under reversal, mirroring and connected sum for convenience.

\section{Definitions}
\label{sec:defs}

\subsection{Klein graphs and unknotting numbers}
\label{subsec:klein}

\begin{defn}[Klein graphs]
A \emph{Klein graph} $\Gamma$ is a finite trivalent graph properly embedded in $S^3$ with a 3-coloring: each edge is colored either red, blue, green, each vertex of the graph is incident to edges of all three colors, and the graph contains all three colors. 
We denote red, blue, green by $r$, $b$, $g$, respectively, and often write $i,j,k$ for arbitrary distinct colors in $\{r,g,b\}$. 
The name arises from the fact that for a given Klein graph there is a regular orbifold cover of $S^3$ with branch set the Klein graph and with deck group the Klein group $\Z/2\Z \times \Z/2\Z$. A Klein graph is \emph{3-Hamiltonian} if for any two distinct colors the union of all edges of those colors is a single cycle. (Observe that a 3-Hamiltonian Klein graph is connected.) 
Klein graphs may have `knot components,' i.e. components without vertices. 
We will assume throughout that Klein graphs are smoothly embedded (except at vertices) and work with surfaces up to diffeomorphism. Throughout, we let $|V(\Gamma)|$ denote the number of vertices of $\Gamma$. The number of edges of $\Gamma$ is equal to $3|V(\Gamma)|/2$ and so the Euler characteristic $\chi(\Gamma) = -|V(\Gamma)|/2$.
\end{defn}

\begin{defn}[Bicolored links, component count]
    Suppose $\Gamma$ is a Klein graph.  For each color $i$, let $\Gamma_i$ be the union of the edges colored $i$.  For each pair of distinct colors $i,j$, the {\em bicolored link} $\Gamma_{ij}$ is (the closure of) the union $\Gamma_i \cup \Gamma_j$.  Observe that $\Gamma$ is $3$-Hamiltonian if and only if each bicolored link is a knot.

    Define the {\em component count}  $\mu(L)$ of a link $L$ to be its number of components and the {\em component count} $\mu(\Gamma)$ of a Klein graph $\Gamma$ to be the total number of components of its bicolored links.  Hence $\mu(\Gamma) = \mu(\Gamma_{rb})+\mu(\Gamma_{bg})+\mu(\Gamma_{rg})$.
\end{defn}

\begin{defn}[Crossing changes]
\label{defn:crossing changes}
A crossing change on a Klein graph can occur either between edges of the same color or between two distinct colors.  We call these types of crossing changes respectively a \emph{same-color crossing change} and a \emph{mixed-color crossing change}, or more simply a \emph{same crossing change} and a \emph{mixed crossing change}.
For a sequence of crossing changes on a Klein graph, we let $s$ denote the number of same crossing changes and $m$ the number of mixed crossing changes. The \emph{Klein Gordian distance} $d_{\Klein}(\Gamma_1, \Gamma_2)$ between two Klein graphs is the minimum of $s + m/2$ over all sequences of crossing changes converting $\Gamma_1$ into $\Gamma_2$, or $\infty$ if there is no such sequence. 

The \emph{Klein unknotting number} $u_{\Klein}(\Gamma)$ of a Klein graph $\Gamma$ is $d_{\Klein}(\Gamma, U)$ where $U$ is a Klein graph embedded in a 2-sphere in $S^3$. Of course, the number $u_{\Klein}(\Gamma)$ is finite if and only if $\Gamma$ is abstractly planar. (If $\Gamma$ is abstractly planar, then by \cite{Mason-homeocurves}, up to isotopy in $S^3$, there is a unique such $U$.) 
\end{defn}

\begin{defn}[Connected sums]
Let $\Gamma_1$ and $\Gamma_2$ be two Klein graphs in distinct copies of $S^3$. 

An {\em order $2$ connected sum} or {\em edge connected sum} $\Gamma_1 \#_2 \Gamma_2$ is obtained as follows. Choose a point in the interior an edge of each graph of the same color and delete  the interior of a closed ball about each point meeting its graph in an arc in the interior of its edge. Glue together the remnants by an orientation reversing homeomorphism of the resulting two $S^2$ boundary components so that the pairs of points of $\Gamma_1$ and $\Gamma_2$ in those spheres match up.  The result $\Gamma_1 \#_2 \Gamma_2$ is again a Klein graph.  Note that the result depends on which edges contain the points and how the pairs of points in the identified spheres match up. 

An {\em order $3$ connected sum} or {\em vertex connected sum} $\Gamma_1 \#_3 \Gamma_2$ is obtained similarly:  Here we choose a vertex of each graph and  delete  the interior of a closed ball about each of these vertices meeting its graph in three proper subarcs of the edges meeting that vertex.  Now glue together the remnants by an orientation reversing homeomorphism of the resulting two $S^2$ boundary components so that the triples of points of $\Gamma_1$ and $\Gamma_2$ in those spheres match up according to their colors.  The result $\Gamma_1 \#_3 \Gamma_2$ is again a Klein graph.  Note that the result depends only on which vertices are chosen \cite{Wolcott}. 

In Section~\ref{sec:littleorientations} we introduce a notion of  {\em total orientation} on a Klein graph. 
 When two totally oriented Klein graphs have total orientations that match on a pair of edges or are opposite on a pair of vertices, these connected sum operations may be performed along the pair in a manner respecting the total orientations in the obvious way to confer a total orientation on the resulting sum.
\end{defn}

\subsection{Orientations and Foams} \label{sec:littleorientations}

Here we give a first pass at defining a kind of orientation for Klein graphs and foams that will enable an Euler characteristic count in Definition~\ref{defn:chi-orb} that is analogous to the slice genus of an oriented link.   In Section~\ref{sec:totalorientations} we more fully develop this orientation and its implications. 

\begin{defn}[Klein graph: Total orientation, double orientation]
\label{def:total orientation definition graph}
A {\em total orientation} on a Klein graph $\Gamma$ is a choice of orientation on each of its bicolored links.  A total orientation on $\Gamma$ induces a {\em double orientation} on each of its edges; that is, an $i$-colored edge inherits an $ij$-orientation from the orientation on $\Gamma_{ij}$ and an $ik$-orientation from the orientation on $\Gamma_{ik}$.    If the $ij$- and $ik$-orientations on an $i$-edge are coherent, we say the doubly oriented edge has a {\em parallel} double orientation and its sign is $+$.  Otherwise the doubly oriented edge has an {\em antiparallel} double orientation and its sign is $-$. 
\end{defn}
We denote mirroring and reversal by $\mir(\Gamma)$ and $-\Gamma$, respectively, and for emphasis, sometimes indicate totally oriented objects with an arrow, \textit{e.g.} $(\vv{\Gamma})$.
\begin{defn}[Foams]\label{defn:foam}
    A \emph{foam (with boundary)} is a properly embedded smooth two-dimensional CW-complex $F$ in a four manifold (here $B^4$ or $S^3\times I$) in which every point in the interior or the boundary is a regular point or a singular point, meaning it has a neighborhood diffeomorphic to one of the pictures in Figure \ref{fig:nbhds}. The singular points $s(F)$ form a 1--complex, called the \emph{seam points}.  Interior vertices of $s(F)$ are called \emph{seam vertices}. 
 (The boundary of $s(F)$ is the collection of univalent vertices of $s(F)$, and these are the vertices of the trivalent graph $\bdry F$.) The \emph{facets} of a foam are the connected components of $F-s(F)$, consisting of sets of regular points, and are smooth surfaces. 
 
 A \emph{Klein foam (with boundary)} is a foam with a 3-coloring whose boundary is a union of Klein graphs. That is, each facet of the foam is colored either red, blue, green and each edge of $s(F)$ is incident to facets of all three colors.  In a Klein foam with boundary the vertices of the bounding Klein graphs are the end points of $s(F)$, and the edges are part of $F-s(F)$.
\end{defn}

\begin{defn}[Klein foam: Total orientation, double orientation]
\label{def:total orientation definition foam}
A Klein foam $F$ is {\em totally orientable} (in \cite{KR}, this is termed {\em admissible}) if each of its bicolored surfaces are orientable.   As with a Klein graph $\Gamma$ and its edges, a {\em total orientation} on a totally orientable $F$ is a choice of orientation on each bicolored surface, and this induces a {\em double orientation} on each of its facets.

Further observe that a total orientation on a Klein foam $F$ with boundary $\Gamma = \bdry F$ induces a total orientation on the Klein graph $\Gamma$.
\end{defn}

\begin{defn}[Foamy cobordisms]\label{defn:foamycobordism}
Let $F$ be a Klein foam that is properly embedded in $S^3 \times [0,1]$. Let $\bdry F = \Gamma_0 \cup -\mir\Gamma_1$, where $\Gamma_0\subset S^3\times \{0\}$ and $\Gamma_1\subset S^3\times \{1\}$ are each Klein graphs or the empty set. We call such a foam a \emph{foamy cobordism}. Note that $s(F)$ is a 1-manifold with boundary if and only if $F$ has no seam vertices. Three particular instances of a foamy cobordism are 
\begin{itemize}
    \item a \emph{slice foam} for $\Gamma_0$, in which $\Gamma_1$ is the empty set but $\Gamma_0$ is not,
    \item a \emph{spanning foam} for $\Gamma_0$, which is a slice foam without seam vertices, and 
    \item a \emph{seamed foamy cobordism} from $\Gamma_0$ to $\Gamma_1$, which is a foamy cobordism without seam vertices in which each edge component of $s(F)$ has one endpoint lying on each of $\Gamma_0$ and $\Gamma_1$;  closed components of $s(F)$ are permitted.  
\end{itemize}
By attaching a $4$-ball to $S^3 \times\{1\}$, spanning foams and slice foams may be regarded as being properly embedded in $B^4$.
\end{defn}

\begin{figure}
    \centering
    \includegraphics[width=.5\textwidth]{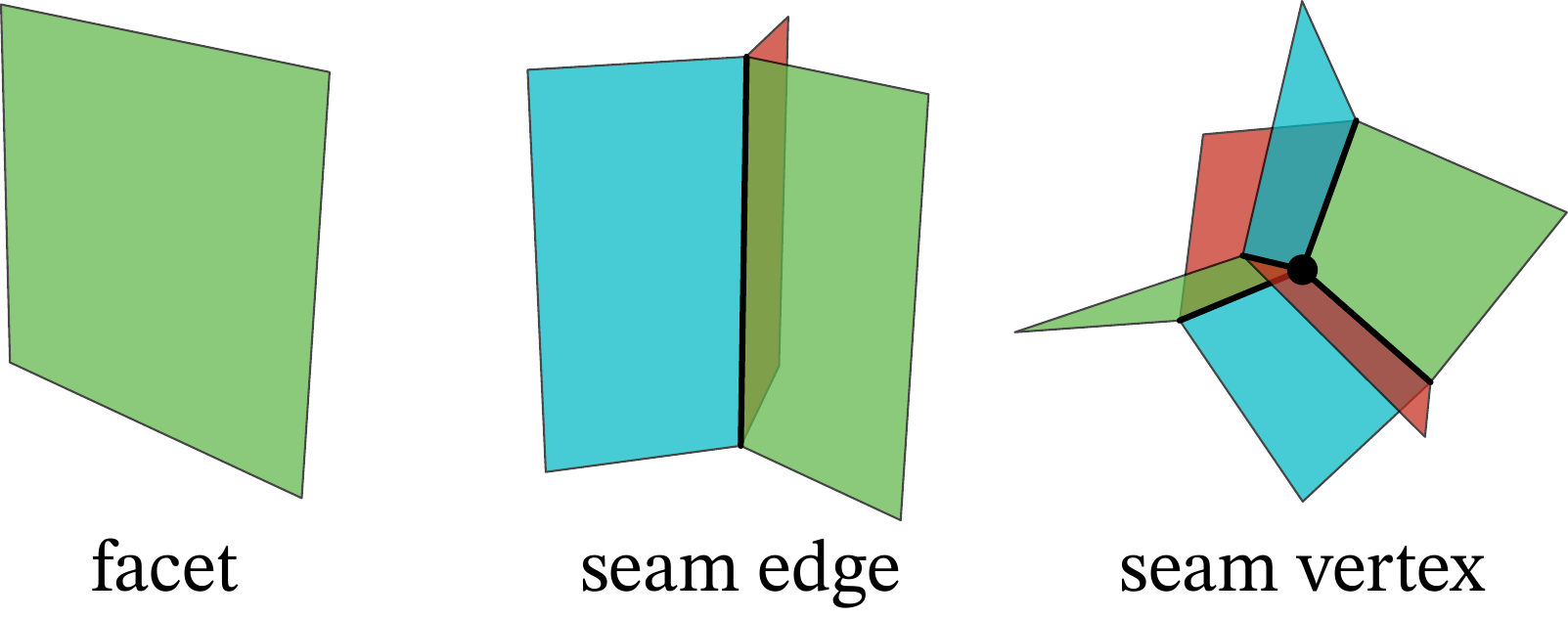}
    \caption{Neighborhoods of a point in a facet, a point in a seam edge, and a seam vertex.}
    \label{fig:nbhds}
\end{figure}

\begin{defn}[Bicolored surfaces]
Suppose $F$ is a Klein foam.  For each color $i$, let $F_i$ be the union of the facets colored $i$.  For each pair of distinct colors $i,j$, the {\em bicolored surface} $F_{ij}$ is the closure of $F_i \cup F_j$. Observe that $F_{ij}$ is indeed a compact surface and contains $s(F)$, but may be disconnected.  A \emph{bicolored sphere} of a Klein foam $F$ is a spherical component of some bicolored surface. When a Klein foam is connected, any closed component of a bicolored surface necessarily contains both colors.  Note that a monocolored sphere (a sphere component of a single color) is a bicolored sphere by definition.
\end{defn}

\begin{defn}[The quantity $|\sv|$]
\label{defn:svabs}
For a Klein graph $\Gamma$ in $S^3$, let $|\sv|(\Gamma)$ be the minimum number of seam vertices among the slice foams bounded by $\Gamma$.  If $\vv{\Gamma}$ is $\Gamma$ with a total orientation, let $|\sv|(\vv{\Gamma})$ be the minimum number of seam vertices among the totally oriented slicing foams bounded by $\vv{\Gamma}$.

Similarly, for a pair of Klein graphs $\Gamma_0$ and $\Gamma_1$, let $|\sv|(\Gamma_0, \Gamma_1)$ be the minimum number of seam vertices among the foamy cobordisms from $\Gamma_0$ to $\Gamma_1$. Analogously define $|\sv|(\vv{\Gamma_0},\vv{\Gamma_1})$ when the two graphs are totally oriented.
\end{defn}

\begin{remark}\label{rem:svlift}
In fact, as we will see in Section~\ref{sec:totalorientations}, seam vertices of totally oriented foams have two isomorphism types.  This observation with a choice which to call `positive' lifts $|\sv|$ to a $\Z$ valued function $\sv$ on totally oriented Klein graphs so that the absolute value of $\sv(\vv{\Gamma})$ is $|\sv|(\vv{\Gamma})$.
In Lemma~\ref{lem:sv-combinatorics}, we show that $\sv(\vv{\Gamma})$ is easily determined from the combinatorics of the totally oriented Klein graph.  
\end{remark}

\begin{lemma}\label{lem:existenceofTO+noBCS2}
    Given a totally orientable Klein foam $F$, there exists a totally orientable Klein foam $F'$ without bicolored spheres such that $\bdry F= \bdry F'$.
\end{lemma}

\begin{proof}
    Let $F$ be a totally orientable Klein foam. Since $F$ is compact, it has at most finitely many bicolored spheres.
    For each bicolored sphere of $F$, choose a facet of $F$ in that sphere and connect sum a torus onto its interior.  The total orientation of $F$ extends across these connected sums. Thus the resulting foam $F'$ continues to be totally orientable with $\bdry F = \bdry F'$ yet has no bicolored spheres.
\end{proof}

\subsection{Orbifold Euler Characteristics}
\label{subsec:orbifolds}

\begin{defn}
\label{defn:chi-orb} \phantom{x}

\begin{itemize}
    \item For a Klein foam $F$, its {\em orbifold Euler characteristic} is
\[\chi^{orb}(F) = \frac{1}{2}\chi(F) - \frac{1}{4}\chi(s(F)).\]

\item 
For a totally oriented Klein graph $\vv{\Gamma}$, its {\em slice orbifold Euler characteristic} is 
\[ \chi_4^{orb}(\vv{\Gamma}) = \max \chi^{orb}(F)\]
where the maximum is taken over all totally oriented slice foams $F$ in $B^4$ with $\vv{\Gamma}=\bdry F$ having $|\sv|(\vv{\Gamma})$ seam vertices and no bicolored spheres. 

\item 
For a Klein graph $\Gamma$ (without a specified total orientation), 
its {\em slice orbifold Euler characteristic} is 
\[ \chi_4^{orb}(\Gamma) = \max \chi^{orb}(F)\]
where the maximum is taken over all totally orientable slice foams $F$ in $B^4$ with $\Gamma=\bdry F$ having $|\sv|(\bdry F)$ seam vertices and no bicolored spheres.

\item For two Klein graphs $\Gamma_1, \Gamma_2$ that cobound a totally orientable seamed foamy cobordism, the {\em seamed cobordism characteristic} is 
\[
\chi_4^{orb}(\Gamma_1, \Gamma_2;s) = \max \chi^{orb}(G)
\]
where the maximum is taken over all totally orientable seamed foamy cobordisms $G$ from $\Gamma_1$ to $\Gamma_2$ having no bicolored spheres.  
\end{itemize}

Lemma~\ref{lem:existenceofTO+noBCS2} gives an easy way to eliminate bicolored spheres.    
Lemmas~\ref{lem:nonHamSliceFoam-simpler} and \ref{lem:noHamSeamCobord-simpler} show that these maxima are bounded above.
\end{defn}

\begin{remark}
As we observe in Lemmas~\ref{lem:notoslicefoam} and \ref{lem:notoseamedfoamycobord}, a Klein graph need not bound a totally orientable slice foam and there need not be a seamed foamy cobordism between two Klein graphs with the same number of vertices.  Nevertheless, as shown in Theorem~\ref{thm:doKleinboundstofoamwithseamvertices}, these always exist if we permit the foams to have seam vertices.
\end{remark}

\begin{remark}\label{rem:orbchi}
\begin{enumerate}
    \item Definition \ref{defn:chi-orb} agrees with standard definitions of the orbifold Euler characteristic for an orbifold $Q$, 
\[
    \chi(Q) =\sum_{q\in Q}\frac{(-1)^{\dim q}}{|G(q)|},
\]
where $G(q)$ is the local group at the open cell $q$ and the sum is over all open cells \cite{Cooper}. In our case, the local group of open 2-dimensional cells of $F$ and the open edges of $\Gamma$ is $\Z/2\Z$, acting by the involution of the double cover. The local group of the seams (including seam vertices) is the entire Klein group.

\item Permitting  arbitrarily many bicolored spheres allows one to construct totally orientable foams having the same boundary yet with arbitrarily large $\chi_4^{orb}$ as follows. Let $F$ be any totally orientable Klein foam and let $F_0$ be the boundaryless Klein foam in $S^4$ that is a suspension of a theta curve. Letting $F'$ be a connected sum of $F$ and $F_0$ along points in the interiors of a facets of $F$ and $F_0$ of the same color, as shown in Figure~\ref{fig:bubbling}(Left), we see that $\chi_4^{orb}(F') = \chi_4^{orb}(F)+\frac12$.

\item Permitting arbitrarily many seam vertices similarly allows one to construct totally orientable foams having the same totally oriented boundary yet with arbitrarily large $\chi_4^{orb}$.  Let $F$ be any totally orientable Klein foam and let $F_0$ be the boundaryless Klein foam in $S^4$ that is the suspension of the tetrahedral graph.  Note that $F_0$ has two seam vertices.   Letting $F'$ be the connected sum of $F$ and $F_0$ along points in the interior of seam edges of $F$ and $F_0$, as shown in Figure~\ref{fig:bubbling}(Right), we see that $\chi_4^{orb}(F') = \chi_4^{orb}(F)+1$.  Also note that if $F$ has no bicolored spheres, then $F'$ doesn't either.  
\end{enumerate}
\end{remark}

\begin{figure}
    \centering
    \includegraphics[width=.9\textwidth]{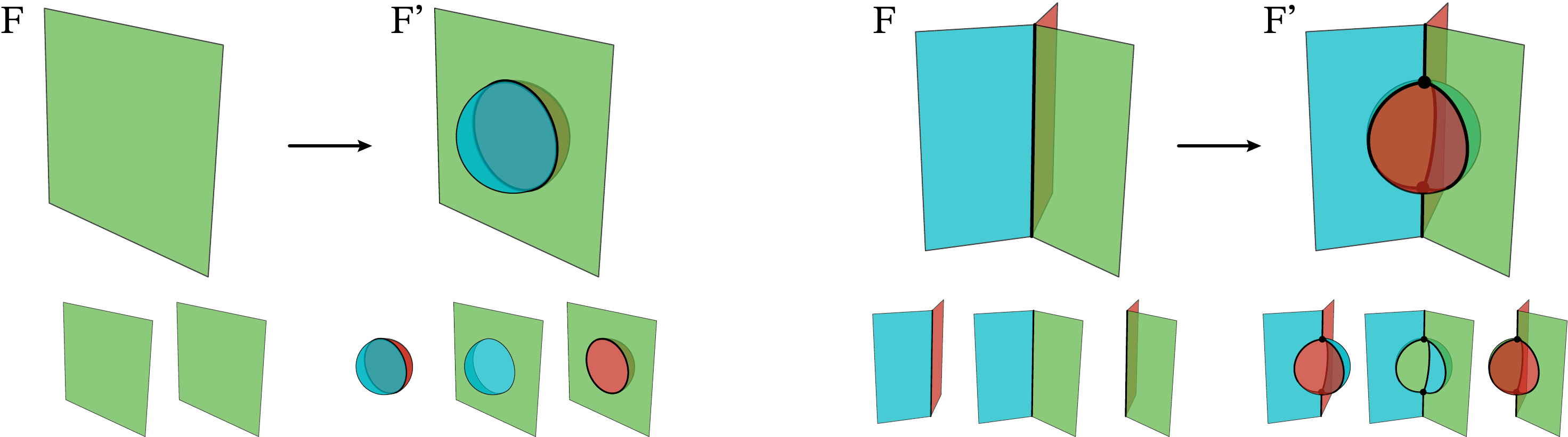}
    \caption{(Left) A portion of a facet of $F$ and then the result $F'$ of its connected sum with the suspension of the theta graph. (Right)  A neighborhood of a point in a seam edge of $F$ and then the result $F'$ of its connected sum with the suspension of the tetrahedral graph. }
    \label{fig:bubbling}
\end{figure}

\begin{lemma}\label{lem:chiorbcounts}
For a foam $F$ without seam vertices, we have
\[4\chi^{orb}(F) = \left( \chi(F_{rb}) + \chi(F_{bg}) + \chi(F_{rg}) \right) - |V(\bdry F)|\]
\end{lemma}

\begin{proof}
Observe that the sum $\chi(F_{rb}) + \chi(F_{bg}) + \chi(F_{rg})$ counts the Euler characteristic of $s(F)$ three times but of $F-s(F)$ only twice.
Since there are no seam vertices, each component of $s(F)$ is either a circle or an edge with its endpoints as vertices of the Klein graph $\bdry F$.   Hence $\chi(s(F)) = \frac12 |V(\bdry F)|$.  Now, beginning with Definition~\ref{defn:chi-orb}, we may count
\begin{align*}
4\chi^{orb}(F) &= 2\chi(F) - \chi(s(F))\\ 
&= 2 \chi(F-s(F)) + \chi(s(F))\\
&= \left(2\chi(F-s(F))+3\chi(s(F)) \right) - 2\chi(s(F))  \\
&= \left( \chi(F_{rb}) + \chi(F_{bg}) + \chi(F_{rg}) \right) - |V(\bdry F)|.
\end{align*}
\end{proof}

The following two lemmas establish that $\chi^{orb}_4(F)$ and $\chi^{orb}_4(\Gamma_0, \Gamma_1;s)$ are bounded above and thus  $\chi^{orb}_4(\Gamma)$ is well-defined. 

\begin{lemma}\label{lem:nonHamSliceFoam-simpler}
    Suppose $\Gamma$ is a Klein graph that bounds a totally orientable slice foam.  Then for any totally orientable slice foam $F$ without bicolored spheres such that $\Gamma = \bdry F$, 
    $4 \chi^{orb}(F)  \leq  \mu(\Gamma)-|V(\Gamma)|$.
    Consequently, $4\chi^{orb}_4(\Gamma) \leq \mu(\Gamma)-|V(\Gamma)|$.  
    
    Furthermore, if $\Gamma$ has no knot components then $4\chi^{orb}_4(\Gamma) \leq \frac12|V(\Gamma)|$. 
\end{lemma}

\begin{proof}
    Let $F$ be a totally orientable slice foam without bicolored spheres for the Klein graph $\Gamma$.  

By Lemma~\ref{lem:chiorbcounts}, 
\[4\chi^{orb}(F) = \left( \chi(F_{rb}) + \chi(F_{bg}) + \chi(F_{rg}) \right) - |V(\bdry F)|.\]
Observe that since $\chi(F_{ij})$ is the sum of the Euler characteristics of the components of $F_{ij}$ which are orientable surfaces and no component is a sphere (because $F$ has no bicolored spheres), we have $\chi(F_{ij}) \leq |\Gamma_{ij}|$ with equality realized exactly when $F_{ij}$ is a union of disks and tori.  Thus $\chi(F_{rg}) + \chi(F_{bg}) + \chi(F_{rg}) \leq \mu(\Gamma)$.  Hence  $4\chi_4^{orb}(F) \leq \mu(\Gamma) - |V(\Gamma)|$.  Therefore $4\chi_4^{orb}(\Gamma) \leq  \mu(\Gamma) - |V(\Gamma)|$.

If $\Gamma$ has no knot components, then since each component of a bicolored link of $\Gamma$ must involve
at least two vertice, we have $3|V(\Gamma)| \geq 2\mu(\Gamma)$.  Hence $\mu(\Gamma)-|V(\Gamma)| \leq \frac32|V(\Gamma)|-|V(\Gamma)| = \frac12|V(\Gamma)|$.  Thus we have $4\chi_4^{orb}(F) \leq \frac12|V(\Gamma)|$.
\end{proof}

\begin{lemma}\label{lem:noHamSeamCobord-simpler}
    Suppose $\Gamma_0$ and $\Gamma_1$ are Klein graphs each with $V$ vertices and $F$ is a totally orientable seamed foamy cobordism without bicolored spheres between them.
    Then $4 \chi^{orb}(F)  \leq  \mu(\Gamma_0) + \mu(\Gamma_1) -2V$.
    Consequently, $4\chi^{orb}_4(\Gamma_0, \Gamma_1;s) \leq \mu(\Gamma_0) + \mu(\Gamma_1) -2V$.
\end{lemma}
\begin{proof}
    Viewing the cobordism as $F \subset S^3 \times [0,1]$, we may choose a properly embedded arc in $S^3 \times [0,1]$ with endpoints in distinct boundary components that is disjoint from $F$.  The exterior of such an arc is $B^4$ with a proper embedding of $F$ having boundary as the split union $\Gamma_0 \sqcup -\mir{\Gamma_1}$ in $S^3$.
    From Lemma~\ref{lem:nonHamSliceFoam-simpler}, it then immediately follows that
    $4 \chi^{orb}(F)  \leq  \mu(\Gamma_0) + \mu(\Gamma_1) -2V$.
    Consequently, $4\chi^{orb}_4(\Gamma_0, \Gamma_1;s) \leq \mu(\Gamma_0) + \mu(\Gamma_1) -2V$.
\end{proof}

\subsection{Signatures}
\label{sec:sigs}

The {\em weak Euler number} $e(F)$ of a Klein foam is the sum of the relative normal Euler numbers of the bicolored surfaces, $e(F) \coloneqq e(F_{rb})+e(F_{bg})+e(F_{rg})$.  In \cite[Definition 3.10]{GR}, Gille-Robert define the signature of a 3-Hamiltonian Klein graph $\Gamma \subset S^3$ to be 
\[\sigma(\Gamma) = \sigma(W_F) + \frac12e(F)\]
where $F$ is a properly embedded spanning Klein foam for $\Gamma$ in $B^4$, $e(F)$ is its weak Euler number, $W_F$ is the 4-manifold Klein cover of $B^4$ branched over $F$, and $\sigma(W_F)$ is its signature. They then prove the following:

\begin{theorem}[{\cite[Proposition 3.14]{GR}}]\label{thm:signatureforconstituentknots}
Suppose that $\Gamma$ is a 3-Hamiltonian Klein graph with bicolored knots $\Gamma_{ij}$ where $\{i,j,k\} = \{r,g,b\}$. 
Then
\[
\sigma(\Gamma) = \sigma(\Gamma_{rb}) + \sigma(\Gamma_{bg}) + \sigma(\Gamma_{rg})
\]
\end{theorem}


As suggested by the 2nd footnote on page 2 of \cite{GR}, their definition of  the signature of a 3-Hamiltonian Klein graph and Theorem~\ref{thm:signatureforconstituentknots} can be extended to Klein graphs that are not 3-Hamiltonian.  This echos the extension of signatures of knots to links where orientation considerations are important.  Let us review some of these details.

Let $\vv{L}$ be a link $L$ endowed with an orientation.  Given any oriented Seifert surface $F$ for $\vv{L}$, the signature $\sigma(\vv{L})$ of $\vv{L}$ is the signature of the symmetrized Seifert pairing $V+V^T$ of $F$.  Equivalently, after isotoping $F$ rel-$\bdry$ to be properly embedded in $B^4$,  $\sigma(\vv{L})=\sigma(M_F)$, the signature of the double of $B^4$ branched over $F$. 

\begin{defn}[Link nullity]
\label{defn:nullity}
The {\em nullity} of the link $\beta({L})$ is the nullity of $(V+V^T)$ plus $r-1$, where $r$ is the number of connected components of $F$, as in \cite{powell}.  Note that the nullity of a link is independent of choices of orientation on the link \cite[Theorem 2.1]{KauffmanTaylor}.  
Note that $0 \leq \beta(\vv{L}) \leq \mu(L)-1$ for any oriented link $\vv{L}$, and in particular, $\mu(L)-1 = 0$ and $\beta(L)=0$ for a knot, where $\mu(L)$ is the number of components of the link. 
While nullity is originally defined as the nullity of $(V+V^T)$ plus $r$, see \cite{Murasugi} and \cite[Remark 1.2]{KauffmanTaylor}, the version of \cite{powell} makes it additive under connected sum. 

\end{defn}

\begin{theorem}[{\cite[Theorem 2.2]{KauffmanTaylor}}]
Let $L_1$ and $L_2$ be links in $S^3$.  Then:
    \[\beta(L_1 \sqcup L_2) = \beta(L_1) + \beta(L_2) + 1 \quad \mbox{ and }\quad \beta(L_1 \# L_2) = \beta(L_1) + \beta(L_2)\]
\end{theorem}

Murasugi \cite{Murasugi} defines the {\em total linking number} $\lambda(\vv{L})$ of $\vv{L}$ to be the sum of the linking numbers of all unordered pairs of components of $\vv{L}$ and then shows that the sum $\sigma(\vv{L})+\lambda(\vv{L})$ is independent of the choice of orientation on $L$.  Hence he defines the signature of the unoriented link $L$ to be $\zeta(L) = \sigma(\vv{L})+\lambda(\vv{L})$. (See also Definition 3.12 and Remark 3.13 of \cite{GR}.)  Gordon-Litherland show that if $F$ is any spanning surface for $L$, then $\zeta(L) = \sigma(M_F)+e(F)$ \cite[Corollary 5']{GordonLitherland}.  
Hence, equipping $L$ with an orientation as $\vv{L}$,
we have $\sigma(\vv{L}) = \sigma(M_F)+e(F) - \lambda(\vv{L})$ regardless of any orientation on $F$.

The following theorem is suggested  but not explicitly stated in \cite[Footnote 2]{GR}.
\begin{theorem}\label{thm:signatureofklein}
    Let $\Gamma$ be a Klein graph in $S^3$.  Let $F$ be a spanning foam for $\Gamma$ in $B^4$.
    Then 
    \begin{itemize}
        \item the integer $\zeta(\Gamma) \coloneqq \sigma(W_F) + \frac12 e(F)$ depends only on $\Gamma$, and
        \item $\zeta(\Gamma) = \zeta(\Gamma_{rb})+\zeta(\Gamma_{bg})+\zeta(\Gamma_{rg})$.
    \end{itemize}
    We say $\zeta(\Gamma)$ is the signature of $\Gamma$.
\end{theorem}

\begin{proof}
    The first part, that the sum $\sigma(W_F) + \frac12 e(F)$ is independent of the choice of spanning foam $F$ for $\Gamma$,  follows exactly as in the proof of \cite[Theorem 3.11]{GR} (skipping the middle part). There \cite[Proposition 3.7]{GR} continues to apply because $W=B^4$ in this situation.  

    The second part follows the first half of the proof of \cite[Theorem 3.14]{GR} where \cite[Theorem 2]{GordonLitherland} extends directly to links as discussed in \cite[\S5]{GordonLitherland} while \cite[Theorem 5']{GordonLitherland} is used in the place of \cite[Theorem 5]{GordonLitherland}.
\end{proof}

\begin{defn}[Total linking number, and nullity]\label{defn:sigTO}
    Let $\vv{\Gamma}$ be a totally oriented Klein graph.  Then the {\em total linking number} of $\vv{\Gamma}$ is the sum of the total linking numbers of its bicolored links. 
 That is, 
 \[ \lambda(\vv{\Gamma})=\lambda(\vv{\Gamma}_{rb})+\lambda(\vv{\Gamma}_{bg})+\lambda(\vv{\Gamma}_{rg}).
 \]
    In accordance with the relation $\zeta(L) = \sigma(\vv{L})+\lambda(\vv{L})$ between the Murasugi signature of an unoriented link $L$ and the signature and total linking number, the {\em signature} of $\vv{\Gamma}$ is then  $\sigma(\vv{\Gamma}) = \zeta(\Gamma)-\lambda(\vv{\Gamma})$.  
Similarly, 
define the {\em nullity} $\beta({\Gamma}) = \beta({\Gamma}_{rb}) + \beta({\Gamma}_{bg}) + \beta({\Gamma}_{rg})$ to be the sum of the nullities of the bicolored links.

It follows immediately from these definitions that the operations of mirroring and total orientation reversal of a totally oriented Klein graph leave $\mu$ the component count and $\beta$ unchanged while $\lambda$ is negated by mirroring but unchanged by total orientation reversal. This is recorded in the following lemma.
\end{defn}

\begin{lemma}\label{lem:invts_of_rev_mir}
    Let $\vv{\Gamma}$ be a totally oriented Klein graph.  Then
    \[
    \begin{array}{rclcll}
\lambda(-\vv{\Gamma}) &=&\lambda(\vv{\Gamma}),&&& \text{ and}\\
\lambda(\mir(\vv{\Gamma}))&=&-\lambda(\vv{\Gamma}),&&& \text{ while}\\
\mu(-\vv{\Gamma})&=&\mu(\mir(\vv{\Gamma}))&=&\mu(\vv{\Gamma}),& \text{ and }\\
\beta(-\vv{\Gamma})&=&\beta(\mir(\vv{\Gamma}))&=&\beta(\vv{\Gamma}).&\\

    \end{array}
    \]
\end{lemma}

\begin{lemma}\label{lem:invts_of_sums}
Let $\Gamma_1$ and $\Gamma_2$ be Klein graphs.
Then we have
\[\beta({\Gamma_1} \#_2 {\Gamma_2}) =  \beta({\Gamma_1}) + \beta({\Gamma_2})+1 \quad \mbox{ and }\quad \beta({\Gamma_1} \#_3 {\Gamma_2}) =  \beta({\Gamma_1}) + \beta({\Gamma_2}) \]
as well as 
\[\mu({\Gamma_1} \#_2 {\Gamma_2}) =  \mu({\Gamma_1}) + \mu({\Gamma_2})-2 \quad \mbox{ and }\quad \mu({\Gamma_1} \#_3 {\Gamma_2}) =  \mu({\Gamma_1}) + \mu({\Gamma_2})-3 \]
\end{lemma}
\begin{proof}
    The statement for the nullity $\beta$ follows from the fact that, as defined, the nullity of a Klein graph is the sum of the nullities of its bicolored links.  Note that an order $2$ sum, say along a red edge, causes a split union of the two blue-green links. The statement for the component count $\mu$ is a straightforward count minding how many bicolored link components are merged under a sum.
\end{proof}

 \begin{remark}
 Observe that for a $3$-Hamiltonian graph $\vv{\Gamma}$ we have $\lambda(\vv{\Gamma})=0$,  $\beta(\vv{\Gamma})=0$, and $\mu(\vv{\Gamma}) = 3$.  Moreover, $\sigma(\vv{\Gamma})$ is independent of the total orientation.   
 \end{remark}   

Armed with Definition~\ref{defn:sigTO}, Theorem~\ref{thm:signatureofklein} immediately allows for the calculation of the signature of a totally oriented Klein graph as a sum of the signatures of its oriented bicolored links.

\begin{cor}\label{cor:signatureofTOKline}
    Let $\vv{\Gamma}$ be a totally oriented Klein graph.  
    Then 
    \[\sigma(\vv{\Gamma}) = \sigma(\vv{\Gamma}_{rb})+\sigma(\vv{\Gamma}_{bg})+\sigma(\vv{\Gamma}_{rg}).\] 
\end{cor}

\begin{proof}
    \begin{align*}
        \sigma(\vv{\Gamma}) &= \zeta(\Gamma)-\lambda(\vv{\Gamma}) \\
         &= \left(  \zeta(\Gamma_{rb})+\zeta(\Gamma_{bg})+\zeta(\Gamma_{rg}) \right) - \left( \lambda(\vv{\Gamma}_{rb})+\lambda(\vv{\Gamma}_{bg})+\lambda(\vv{\Gamma}_{rg})  \right) \\
         &= \left(  \zeta(\Gamma_{rb})- \lambda(\vv{\Gamma}_{rb})  \right) + \left(  \zeta(\Gamma_{bg})- \lambda(\vv{\Gamma}_{bg}) \right) + \left(  \zeta(\Gamma_{rg})- \lambda(\vv{\Gamma}_{rg}) \right)  \\
         &= \sigma(\vv{\Gamma}_{rb})+\sigma(\vv{\Gamma}_{bg})+\sigma(\vv{\Gamma}_{rg})
    \end{align*}\end{proof}

\begin{remark}
In \cite[Proposition 3.14]{GR} (which we have only partially stated in Theorem~\ref{thm:signatureforconstituentknots}) the hypothesis that the Klein graph $\Gamma$ is 3-Hamiltonian is indeed used for the {\em strong signature} $\tilde{\sigma}(\Gamma)$; see Section~\ref{subsec:strong sig}.  Their definition of this strong signature uses that the double branched cover of each bicolored knot $\Gamma_{ij}$ is a rational homology sphere to ensure that the lift of the third color of edges $\Gamma_k$ is a link of rationally null-homologous knots.  The double branched cover of a knot is necessarily a rational homology sphere while the double branched cover of a link of more than one component might not be.
\end{remark}

\section{Total orientations of Klein graphs and seam vertices}
\label{sec:totalorientations}

\subsection{Foams with seam vertices}
\label{subsec:foams with seams}

For any spatial Klein graph $\Gamma$ in $S^3$, Gille-Robert show how to construct a spanning foam (ie.\ a Klein foam without seam vertices) in $B^4$ whose boundary is $\Gamma$ \cite[Proposition 2.4]{GR}.  However, some Klein graphs  do not bound any totally orientable slice foam as we now observe.

Let $\Gamma_{tet} \subset S^3$ be the (planar) tetrahedral Klein graph shown in Figure~\ref{fig:seamvertex}. 
Define $F_{vtx}$ to be the Klein foam that arises from the standard neighborhood $(B^4, F_{vtx})$ of a seam vertex of a Klein foam, a cone on $(S^3, \Gamma_{tet})$. 
Also note that $\Gamma_{tet}$ is $3$-Hamiltonian.

\begin{figure}
    \centering
    \includegraphics[width=.6\textwidth]{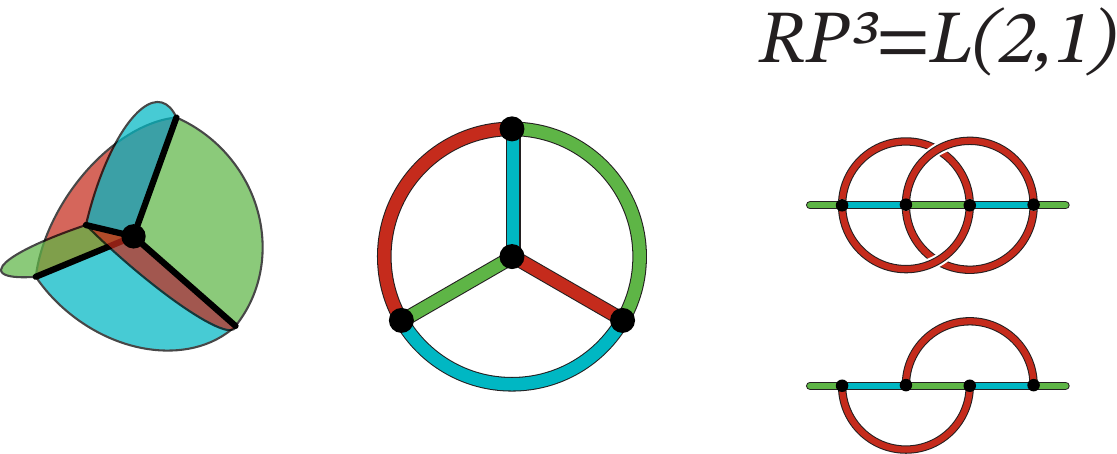}
    \caption{(Left) A neighborhood of a seam vertex. (Center) The planar (trivial) Klein tetrahedral graph is the boundary of a regular neighborhood of a seam vertex. (Right) The Klein cover of tetrahedral graph.}
    \label{fig:seamvertex}
\end{figure}

\begin{lemma}\label{lem:notoslicefoam}
    While $F_{vtx}$ is totally orientable,
    $\Gamma_{tet}= \bdry F_{vtx}$ does not bound a totally orientable Klein foam without seam vertices.
\end{lemma}

\begin{proof}
    Suppose $F$ is a totally orientable Klein foam for $\Gamma_{tet}$ without seam vertices. 
    Since $\Gamma_{tet}$ has just four vertices, $F$ has exactly two seam edges.  By symmetry, we may assume each seam edge has the same endpoints as a red edge. 
    Then the two blue edges and two seam edges form a single cycle as do the two green edges and two seam edges.  Thus the facets $F_b$ and $F_g$ are each connected with connected boundary.  

    Since $F$ is totally orientable, the surface $F_{bg}$ is orientable.   Choose an orientation of $F_{bg}$. This induces an orientation on each $F_b$ and $F_g$.  Moreover, the $bg$-orientations on the blue and green edges of $\Gamma_{tet}$ arising from $\bdry F_{bg}$ are also induced from the boundary orientations on $\bdry F_b$ and $\bdry F_g$ 
    However, with these $bg$-orientations, one seam edge connects the tips of the blue edges while the other connects their tails.  Yet this is inconsistent with any orientation induced from $\bdry F_b$.
\end{proof}

\begin{remark}\label{rem:TOspanningfoamsforonlysomeTOs}
    The Klein graph $\Gamma_{tet} \#_2 \Gamma_{tet}$ bounds a totally orientable slice foam only for certain total orientations while other total orientations only bound totally oriented foams with at least two seam vertices.
\end{remark}

\begin{lemma}\label{lem:notoseamedfoamycobord}
    Let $\Gamma_{2\thetan} = \Gamma_\thetan \#_2 \Gamma_\theta$ be the edge connected sum of two trivial theta curves.  There is no totally orientable seamed foamy cobordism from $\Gamma_{tet}$ to $\Gamma_{2\thetan}$, even though they have the same number of vertices.
\end{lemma}

\begin{proof}
    Suppose $G$ were a totally orientable seamed foamy cobordism from $\Gamma_{tet}$ to $\Gamma_{2\thetan}$.  Then a total orientation on $G$ induces a total orientation on $\Gamma_{2\thetan}$.  Observe that $\Gamma_{2\thetan}$ is the boundary of a $2$-complex $H$ in $\R^3 \subset S^3$ formed as the product of an interval with a planar graph in the form of the letter `H'.  With its interior pushed into $B^4$, $H$ is a totally orientable slice foam for $\Gamma_{2\thetan}$.  Moreover, one readily observes that any total orientation on $\Gamma_{2\thetan}$ arises as the boundary of some total orientation on $H$.  Hence, with the total orientation on $\Gamma_{2\thetan}$ induced from the total orientation on $G$, there is a total orientation on $H$ so that $G \cup_{\Gamma_{2\thetan}} H$ is a totally orientable slice foam for $\Gamma_{tet}$.  However this contradicts Lemma~\ref{lem:notoslicefoam}.
\end{proof}

In contrast to Lemmas~\ref{lem:notoslicefoam} and \ref{lem:notoseamedfoamycobord}, if we permit our foams to have seam vertices, we can always find totally orientable foams.

\begin{example}\label{exa:coveroftet}
    The Klein cover of $\Gamma_{tet}$ is $\RP^3$ and
    the Klein cover of $F_{vtx}$ is the cone on $\RP^3$.
In the double branched cover over the unknot $(\Gamma_{tet})_{ij}$, the edges $(\Gamma_{tet})_{k}$ lift to a Hopf link $(\widetilde{\Gamma}_{tet})_{k}$.  Since the Klein cover of $\Gamma_{tet}$ is then the double branched cover of $(\widetilde{\Gamma}_{tet})_{k}$, we see that the Klein cover of $\Gamma_{tet}$ is $\RP^3$.  Since $(B^4, F_{vtx})$ is a cone on $(S^3, \Gamma_{tet})$, this then implies that the Klein cover of $F_{vtx}$ is the cone on $\RP^3$.  See Figure~\ref{fig:seamvertex}.  
\end{example}

\begin{remark}
Observe that  Example~\ref{exa:coveroftet} 
shows that the Klein cover of a foam with seam vertices is not a $4$-manifold but rather a pseudo $4$-manifold where each singularity is a cone on $\RP^3$.    
\end{remark}

\begin{theorem}\label{thm:doKleinboundstofoamwithseamvertices}
    Every totally oriented Klein graph is the boundary of a totally oriented Klein foam, possibly with seam vertices.
\end{theorem}

\begin{figure}
    \centering
    \includegraphics[width=.7\textwidth]{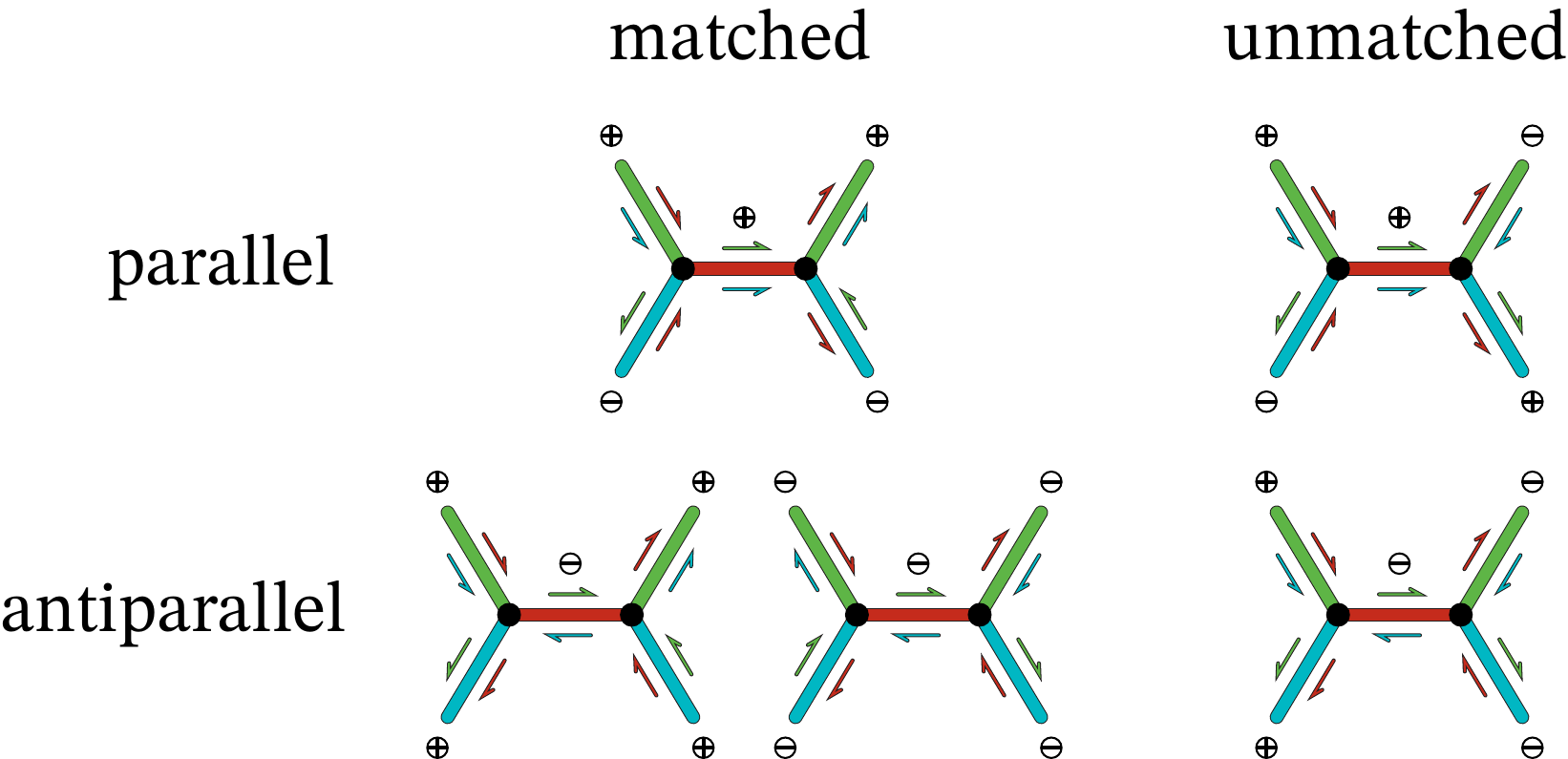} 
    \caption{An edge of a totally oriented Klein graph is either parallel or antiparallel and either matched or unmatched.  Up to symmetry, there are two possibilities for a matched antiparallel edge and one otherwise; these five possibilities are shown for a red edge.}
    \label{fig:doublyorientededges}
\end{figure}

\begin{proof}
    Let $\vv{\Gamma}$ be a totally oriented Klein graph.
    We will induct on the number of vertices of $\vv{\Gamma}$.

    If $\vv{\Gamma}$ has no vertices, then $\vv{\Gamma}$ is a link.  First, by mixed crossing changes, $\vv{\Gamma}$ may be separated into a split link  where each split component is a link of a single color. These mixed crossing changes may be realized by a sequence of foamy cobordisms $F_m$ that respect the double orientations as done in the proof of Theorem~\ref{chi-orb s-m}; one only needs to be mindful of the $ij$-orientations on a mixed crossing change between an $i$-edge and a $j$-edge. 
    Next, two linked components of the same color but with double orientations of opposite sign may be unlinked as follows.   
    Use the elementary open/close cobordism as in Figure~\ref{fig:elementarycobordism}(Top Right) to `slice open' one of the knot components. (This basically connect sums the knot with a trivial theta curve.)  Then use mixed crossing changes with the two new edges to split away the other knot component.  Thereafter the sliced open component can be sewn closed using the elementary open/close cobordism again but in the other direction.    

    Eventually, we have a split link where each split component is a (possibly multi-component) link of a single color and where the two orientations on each component are all parallel or all antiparallel. Each split component bounds a Seifert surface, and we may arrange that the Seifert surfaces for the split components are mutually disjoint. We color each Seifert surface according to the color of its bounding split component and we endow each with two orientations induced from the two orientations on the split component. The result is a foam, bounded by the split link, where the orientations of the bicolored surfaces are coherent, giving a total orientation to the foam. Together with the foams from the previous cobordisms, we obtain a totally oriented foamy cobordism from  the totally oriented, Klein colored link $\vv{\Gamma}$ to the empty set.  Thus $\vv{\Gamma}$ is the boundary of a totally oriented slice foam.
    
    Now assume $\vv{\Gamma}$ has vertices.
    Note that a vertex has an odd number of negative edges incident to it (Definition \ref{def:total orientation definition graph}).  Since $\vv{\Gamma}$ has a vertex, it has an edge of each color.
    
    Let $e$ be a $k$-edge.   Observe that the two $i$-edges adjacent to $e$ have the same sign if and only if the two $j$-edges adjacent to $e$ also have the same sign.  Say such an edge is {\em matched}; otherwise say it is {\em unmatched}.  See Figure~\ref{fig:doublyorientededges}.  
    
    If $e$ is matched, an elementary unzip cobordism, Figure~\ref{fig:elementarycobordism}(Top Center), along $e$ yields a totally oriented foamy cobordism from $\vv{\Gamma}$ to a Klein graph $\vv{\Gamma}'$ with two fewer vertices obtained by unzipping $\vv{\Gamma}$ along $e$.  By induction $\vv{\Gamma}'$ bounds a totally oriented slice foam, and this cobordism extends to a totally oriented slice foam bounded by $\vv{\Gamma}$. 
    
    If $e$ is unmatched, an elementary I-H cobordism, shown in Figure~\ref{fig:elementarycobordism}(Bottom Left), along $e$ yields a totally oriented foamy cobordism from $\vv{\Gamma}$ to a graph $\vv{\Gamma}'$ with the same number of vertices obtained by exchanging $e$ for an edge $e'$. 
    This edge $e'$ has the same set of edges adjacent to it as $e$, but partitioned differently at each of its endpoints.
    This forces $e'$ to have the sign opposite that of $e$. 
    (One may check the two unmatched cases from Figure \ref{fig:doublyorientededges} directly.) 
    Furthermore, if $\vv{\Gamma}$ has no matched edges, then $\vv{\Gamma}'$ will have a matched edge:  If $f$ is an $i$-edge incident to $e$ in $\vv{\Gamma}$, then at its other end is a $k$-edge $g$ with sign opposite that of $e$ since $f$ (and every other edge) is unmatched.  However in $\vv{\Gamma}'$ the edge $g$  has the same sign as $e'$ so that $f$ is now matched.  Thus an elementary unzip cobordism may be performed along $f$ in $\vv{\Gamma}'$ yielding a totally oriented cobordism from $\vv{\Gamma}$ through $\vv{\Gamma}'$ to a Klein graph $\vv{\Gamma}''$ that has two fewer vertices.  By induction $\vv{\Gamma}''$ bounds a totally oriented slice foam, and this cobordism extends it to a totally oriented slice foam bounded by $\vv{\Gamma}$.
\end{proof}

\subsection{Total orientations, vertices, and seam vertices}
\label{sec:totalorientationsverticesandseamvertices}

In the following we show that totally oriented Klein graphs are partitioned into totally oriented seamed foamy cobordism classes by a function $\sv(\bullet)$ that is a lift of the function $|\sv(\bullet)|$.
Indeed it will follow that for any pair of totally oriented Klein graphs $\Gamma_0,\Gamma_1$, any totally oriented foamy cobordism needs at least $|\sv(\Gamma_0)-\sv(\Gamma_1)|$ seamed vertices (of the appropriate RGB/BGR type), and there is such a foam with exactly that many.

\begin{defn}[Vertex total orientations, RGB and BGR types]\label{defn:TOvertextypes}
Locally, a vertex of a Klein graph $\Gamma$ with its three incoming edges admits $8$ total orientations. These may be arranged as the vertices of a cube joined by an edge when one bicolored orientation is reversed. See Figure~\ref{fig:TOvertextypes}.  Reversing the total orientation pairs them up as opposite vertices of the cube.   One pair has three negative edges and the other three pairs have a single negative edge. A vertex with three negative edges has a cyclic ordering as either {\em RGB} or {\em BGR}: Say the ordering is RGB if the bicolored arc orientations go $\vv{rg}$, $\vv{gb}$, and $\vv{br}$ where $\vv{ij}$ indicates that the $ij$-arc is oriented at the vertex with an incoming $i$-edge and and outgoing $j$-edge.  In Figure~\ref{fig:TOvertextypes}, the vertex at the far right has the RGB type while the vertex at the far left has the BGR type.
\end{defn}

\begin{figure}
    \includegraphics[width=.6\textwidth]{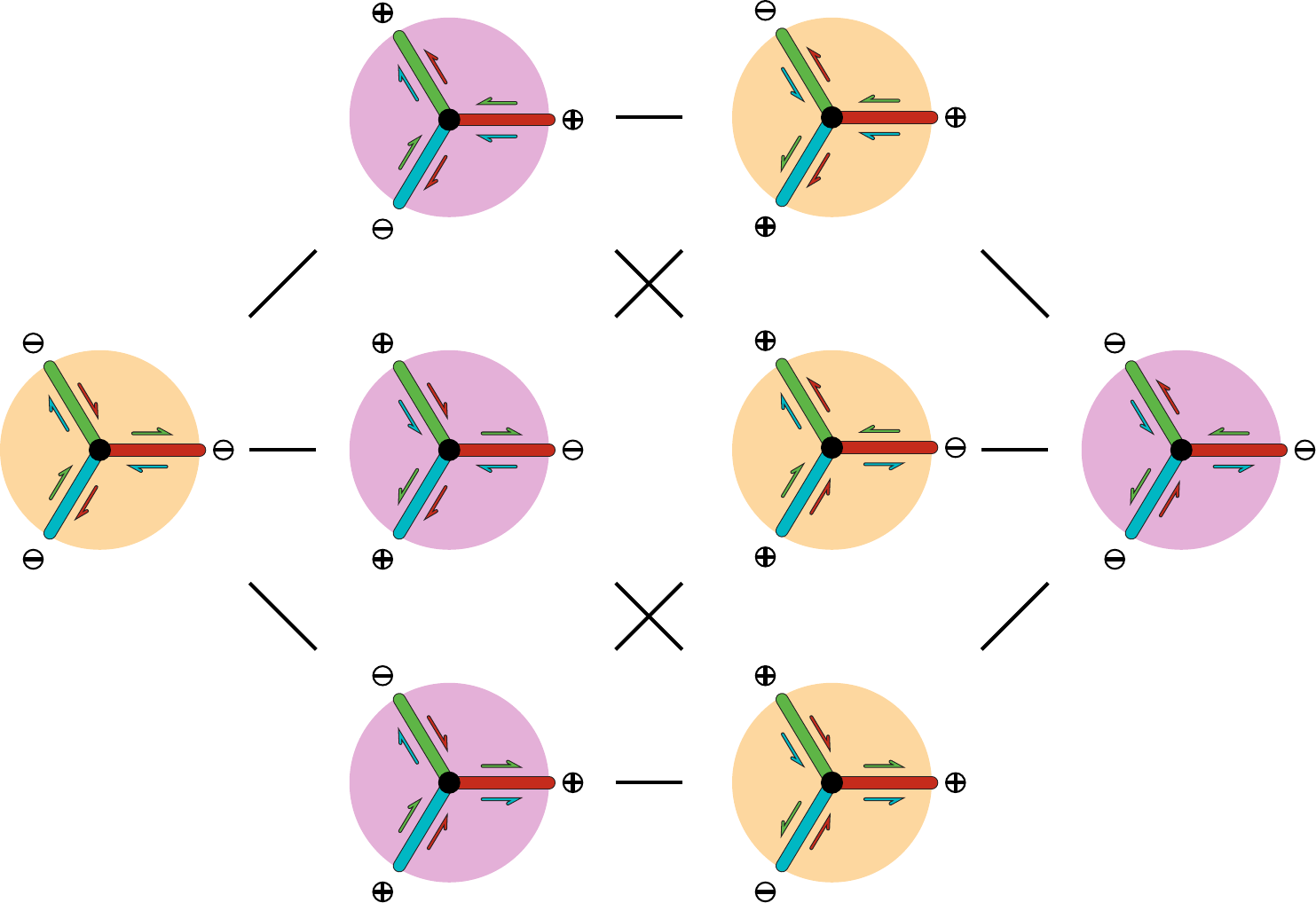}
    \caption{The eight total orientations of vertices in a Klein graph are arranged in a cubic graph.  Adjacent vertices differ by reversal of one bicolored orientation.  Vertices with opposite orientation are at opposite vertices of the cube.  The highlighting shows the partition into the two tetrahedral orientation types.}
    \label{fig:TOvertextypes}
\end{figure}

\begin{defn}[Signed seam vertex count $\sv$]\label{defn:signedsvcount}
The tetrahedral graph $\Gamma_{tet}$ admits $2$ total orientations.  For each total orientation, the four vertices correspond to one set of the four non-adjacent vertices in the cube graph.  See Figure~\ref{fig:TOvertextypes}.  Hence we can also define a totally oriented tetrahedral graph as either RGB or BGR according to the type of its triply negative vertex.  

A regular neighborhood of a seam vertex in a foam is a cone $F_{vtx}$ on the tetrahedral graph.  Hence a seam vertex of a totally oriented foam is also either RGB or BGR according to the totally oriented tetrahedral graph in the boundary of this cone.

For a totally oriented Klein graph $\vv{\Gamma}$, let $RGB(\vv{\Gamma})$ be the number of $RGB$ vertices and let $BGR(\vv{\Gamma})$ be the total number of $BGR$ vertices.  Then define the {\em signed seam vertex count} to be $\sv(\vv{\Gamma}) = RGB(\vv{\Gamma})-BGR(\vv{\Gamma})$.  
\end{defn}

\begin{lemma}
\label{lem:sv-additive}
The quantity $\sv(\vv{\Gamma})$ is unchanged by mirroring, negated by orientation reversal, and additive under connected sum. 
That is,
\[\sv(\mir{\vv{\Gamma}}) = \sv(\vv{\Gamma}) \qquad \mbox{but} \qquad \sv(-\vv{\Gamma}) = -\sv(\vv{\Gamma})\]
and
\[ \sv(\vv{\Gamma_1} \#_2 \vv{\Gamma_2}) = \sv(\vv{\Gamma_1}) + \sv(\vv{\Gamma_2}) \qquad \mbox{and} \qquad \sv(\vv{\Gamma_1} \#_3 \vv{\Gamma_2}) = \sv(\vv{\Gamma_1}) + \sv(\vv{\Gamma_2}).\]
\end{lemma}
\begin{proof}
The behavior under mirroring and orientation reversal are immediate from the actions of these operations on the signs of the vertices.   The additivity follows fairly directly too as an edge connected sum has the same set of oriented vertices as its summands while a vertex connected sum fuses vertices of opposite sign.  Indeed, a vertex connected sum may be viewed as a tribanding (Definition~\ref{defn:triband} below) of the split union of the two summands.
\end{proof}

\begin{defn}[Triband cobordism]\label{defn:triband}
    Let $\YGraph$  be the graph in $B^3$ that is the cone on $3$ points in $S^2$. Given a Klein graph $\Gamma_0$, {\em triband} on $\Gamma$ is an embedding of $\phi \colon \YGraph \times I \to S^3$  so that 
    \begin{itemize}
        \item the intersection $\phi(\YGraph \times I) \cap \Gamma_0 = \phi(\YGraph \times \bdry I)$ is a closed neighborhood of a pair of vertices, and
        \item  the three arcs of $\phi(\bdry \YGraph \times I)$ connect edges of the same color.
    \end{itemize} 
    Let $\Gamma_1$ be the Klein graph resulting from surgery along the triband.   That is, $\Gamma_1 = (\Gamma_0 - \phi(\YGraph \times \bdry I)) \cup \phi(\bdry \YGraph \times I)$.  We say $\Gamma_1$ is the result of a {\em tribanding} of $\Gamma_0$ along the two vertices.

     In $S^3 \times [0,2]$, the foam $F = \Gamma_0 \times [0,1] \cup \phi(\YGraph \times I) \times \{1\} \cup \Gamma_1 \times [1,2]$ gives the {\em triband cobordism} from $\Gamma_0$ to $\Gamma_1$.  

    Furthermore, if  $\Gamma_0$ is endowed with a total orientation and the tribanding is along a pair of oppositely oriented vertices, then the total orientation extends across the triband cobordism giving a total orientation on $\Gamma_1$ as well. 
\end{defn}

\begin{lemma}\label{lemma:vertexorientations}
        A connected component of a totally oriented Klein graph has either 
        \begin{itemize}
            \item no vertices,
            \item a pair of oppositely oriented vertices, or
            \item all four types of vertices of tetrahedral orientation.
        \end{itemize}
\end{lemma}
    \begin{proof}
        Let $v$ be a vertex of a connected totally oriented Klein graph that has no pair of oppositely oriented vertices. We must then show that the graph has the other three orientation types of vertices of same the tetrahedral orientation of $v$.  Using Figure \ref{fig:TOvertextypes}, for any color $i$, one checks that an $i$-edge may join $v$ to only two types of oriented vertices, one of them has the opposite orientation while the other belongs to the same tetrahedral orientation (shown as purple and orange in the figure). If there were a pair of edges between two vertices they must have opposite orientation, therefore there is at most one edge between any pair of vertices. Thus $v$ must be adjacent to three different vertices. One further sees that the three vertices adjacent to $v$ must all have different orientation types, but belong to the same tetrahedral orientation.
    \end{proof}

\begin{lemma}
\label{lem:sv-combinatorics}
Let $\Gamma$ be a totally oriented Klein graph.   Any totally oriented foam $F$ with $\Gamma$ as its boundary requires at least $|\sv(\Gamma)|$ seam vertices.  Moreover there is a totally oriented foam $F$ with $\bdry F = \Gamma$ with exactly $|\sv(\Gamma)|$ seam vertices and these vertices are $RGB$ if $\sv(\Gamma)>0$ and $BGR$ if $\sv(\Gamma)<0$.
\end{lemma}

\begin{proof}
    Suppose a totally oriented Klein graph $\Gamma$ bounds a totally oriented foam $F$ without seam vertices.
    Using triband cobordisms, one may construct a foamy cobordism $F$ without seam vertices from $\Gamma$ to a totally oriented Klein graph $\Gamma'$ where no pairs of vertices have opposite orientation. Observe that $\sv(\Gamma) = \sv(\Gamma')$.  By a total orientation reversal if needed, assume $\sv(\Gamma') \geq 0$.
    By Lemma~\ref{lemma:vertexorientations}, any component of $\Gamma'$ with a vertex must have a RGB or a BGR vertex an.  So $\Gamma'$ has $\sv(\Gamma')$ RGB vertices (since we have eliminated pairs of vertices of opposite orientation) and in fact $\sv(\Gamma')$ vertices of each other RGB tetrahedral orientation type.

    Using a seam vertex we can cap off a vertex and its three neighbors.  With $|\sv(\Gamma')|$ applications of this, we arrive at a totally oriented Klein graph without vertices. Such graphs bound totally orientable foams without seam vertices.   Put together, $\Gamma$ bounds a totally oriented foam with exactly $|\sv(\Gamma)|$ seam vertices.

    If $\Gamma$ were to bound a totally oriented foam with fewer seam vertices, then so does $\Gamma'$.  However then some triple negative RGB vertex of $\Gamma'$ would have to be the end of a properly embedded seam edge.  At the other end would be a vertex with the opposite orientation, contrary to the construction of $\Gamma'$.
\end{proof}

\subsection{Cobordisms and connected sums of totally oriented Klein graphs}
\label{subsec:cobordisms}

\begin{lemma}
    Suppose $\vv{F}$ is a totally oriented foamy cobordism in $S^3 \times [0,1]$ from $\vv{\Gamma_0}$ to $\vv{\Gamma_1}$ with a spanning seam edge $s$ between vertices $v_i$ of $\vv{\Gamma}_i$. 
 Then $\vv{F}'=\vv{F} \cut N(s)$ is a totally oriented foam in $B^4$ with $\bdry \vv{F}' = \vv{\Gamma_1} \#_3 -m(\vv{\Gamma}_0)$ where $-m(\vv{\Gamma}_0)$ is the reversed mirror of $\vv{\Gamma}_0$ and the connected sum is along the vertices $v_i$.
\end{lemma}
\begin{proof}
    This follows from basically the same proof as for the case of cobordisms between oriented links.
\end{proof}

\begin{lemma}\label{lem:seamtransposition}
    Let $\vv{\Gamma}$ be a totally oriented Klein graph in $S^3$.  Let $v_0$ and $v_1$ be two vertices of the same orientation type.  Then there is a seamed foamy cobordism from $\vv{\Gamma}$ to itself for which the permutation on the vertices induced by the seams is a transposition of $v_0$ and $v_1$.
\end{lemma}

\begin{proof}
    Begin with the identity cobordism $\vv{F} = \vv{\Gamma} \times I$ in $S^3 \times I$.  Let $s_i$ be the seam edge $v_i \times I$, and choose a point $p_i$ in the interior of $s_i$.
    Observe that the boundary $S_i=\bdry N(p_i)$ of a small regular neighborhood $N(p_i)$ is a $3$-sphere that meets $\vv{F}$ in a trivial $\thetan$-curve $\vv{\thetan}_i$ with vertices $S_i \cap s_i$.  The total orientation on $\vv{\thetan}_i$ is induced by the foam $\vv{F} \cap N(p_i)$ in the $4$-ball $N(p_i)$ so that the vertex $v_i^+$ of $S_i \cap s_i$ with higher $I$--coordinate has the same orientation type as $v_i$ while the other $v_i^-$ has the opposite orientation type.  
    
    Now choose a simple arc $\phi$ in $S^3 \times I$ connecting $p_0$ to $p_1$ that is otherwise disjoint from $F$.   Then the ball $N(\phi)$ is the tubing of the two balls $N(p_0)$ and $N(p_1)$ and its boundary $S = \bdry N(\phi)$ is the connected sum of $S_0$ and $S_1$ that meets $F$ in the split graph  $\vv{\thetan}_0 \sqcup \vv{\thetan}_1$.  (The total orientation is still induced from $\vv{F} \cap N(\phi) = \vv{F} \cap (N(p_0) \cup N(p_1))$.) Since $v_0$ and $v_1$ have the same orientation type, $\vv{\thetan}_0$ and $\vv{\thetan}_1$ are isomorphic as totally oriented Klein graphs.  Furthermore, because $\vv{\thetan}_0$ is a trivial $\thetan$-curve, it is isotopic to its own reverse.  Because they are embedded as a split graph, there is a totally oriented Klein foam $\vv{F_\thetan}$ homeomorphic to $\vv{\thetan}_0 \times I$ and embedded in $S$ with $\bdry \vv{F_\theta} = \vv{\thetan}_0 \sqcup \vv{\thetan}_1$ so that the seams of $\vv{F}_\thetan$ connect $v_0^+$ to $v_1^-$ and $v_0^-$ to $v_1^+$.   Replacing $\vv{F} \cap N(\phi)$ with $\vv{F}_\theta$ produces a new totally oriented Klein foam $\vv{F}'$ in which $s_0$ and $s_1$ have been replaced by seams that connect one end of $v_0 \times \bdry I$ to the other end of $v_1 \times \bdry I$ and vice-versa, leaving all other seams as they were and producing the desired permutation.
\end{proof}

\begin{lemma}
\label{lem:seambetweenvtxes}
    Suppose there is a seamed foamy cobordism from $\vv{\Gamma}_0$ to $\vv{\Gamma}_1$.
    Then for any pair of vertices $v_i$ of $\vv{\Gamma}_i$ with the same orientation type, there exists a seamed foamy cobordism from $\vv{\Gamma}_0$ to $\vv{\Gamma}_1$ with seam edge connecting $v_0$ to $v_1$.
\end{lemma}

\begin{proof} 
This follows after applying Lemma~\ref{lem:seamtransposition} to construct a seamed foamy cobordism from $\vv{\Gamma_1}$ to itself yielding a permutation of the vertices $v_0$ and $v_1$. 
\end{proof}

\begin{lemma}
\label{lem:seamedfoamycobord_implies_slicefoam}
    Suppose $\vv{\Gamma}_0$ and $\vv{\Gamma}_1$ are totally oriented Klein graphs.
    For any given pair of vertices $v_i$ of $\vv{\Gamma}_i$ with the same orientation type, form the connected sum $\vv{\Gamma_1} \#_3 -\mir{\vv{\Gamma}_0}$ along $v_0$ and $v_1$. 
    If there is a seamed foamy cobordism between $\vv{\Gamma}_0$ and $\vv{\Gamma}_1$, 
    then there is a slice foam for $\vv{\Gamma_1} \#_3 -\mir{\vv{\Gamma}_0}$. 
\end{lemma}
\begin{proof}
        This proof essentially follows the analogous construction of a slice disk for $K_1 \# -\mir{K_0}$ from a concordance between $K_0$ and $K_1$.
        First, by Lemma~\ref{lem:seambetweenvtxes}, there exists a seamed foamy  cobordism $\vv{F}$ from $\vv{\Gamma}_0$ to $\vv{\Gamma}_1$ with a seam edge $e$ connecting $v_0$ to $v_1$.  Then $\vv{F}\cut\nbhd(e)$ is then a slice foam for $\vv{\Gamma_1} \#_3 -\mir{\vv{\Gamma}_0}$.
\end{proof}

\section{Bounds}
\label{sec:results}
In this section, we prove Theorem \ref{thm:main}, bounding the orbifold Euler characteristic of the graph $\Gamma_1\#_3-\mir{\Gamma_2}$ below by the signature difference, and above by the Klein Gordian distance between $\Gamma_1$ and $\Gamma_2$. 

\begin{theorem}
\label{thm:main}
Suppose that $\vv{\Gamma_1}$ and $\vv{\Gamma_2}$ are totally oriented Klein graphs that are related by a sequence of crossing changes. 
Set $V:=|V(\Gamma_1)|=|V(\Gamma_2)|$ and $\mu:= \mu(\Gamma_1)=\mu(\Gamma_2)$. Then:

\begin{align*}
 |\sigma(\vv{\Gamma_1})-\sigma(\vv{{\Gamma_2}})| - \beta(\vv{\Gamma_1})-\beta(\vv{{\Gamma_2}}) +4\mu  - 12 
 &\labelrel\leq{ineq:a} 
 5- 4 \chi^{orb}_4(\vv{\Gamma_1} \#_3 -\mir{\vv{\Gamma_2}}) -2V \\
 &\labelrel\leq{ineq:b} 
 - 4 \chi^{orb}_4(\vv{\Gamma_1}, \vv{\Gamma_2}; s) -2V\\
 &\labelrel\leq{ineq:c} 
 4d_{\Klein}(\vv{\Gamma_1}, \vv{\Gamma_2})    
\end{align*}
where $\vv{\Gamma_1} \#_3 -\mir{\vv{\Gamma_2}}$ is any vertex connected sum of $\vv{\Gamma_1}$ and $-\mir{\vv{\Gamma_2}}$ compatible with the total orientations.
\end{theorem}

\begin{proof}
    In subsection~\ref{sec:ab}, 
    Corollary~\ref{two ended cobordism} proves inequality  \eqref{ineq:a} since both graphs will have the same signed vertex count $\sv$ and Lemma \ref{lem:slicecharcobordisms} establishes inequality \eqref{ineq:b}. In subsection \ref{sec:cd} we prove inequality \eqref{ineq:c}. The proof then follows.
\end{proof}

\begin{proof}[Proof of Corollary \ref{cor:sigsliceunknot}]
Let $\vv{\Gamma_1}$ be $\thetan$ with any Klein coloring and total orientation.  Let $\vv{\Gamma_2}$ be a totally oriented and Klein colored trivial theta with the same vertex types as $\vv{\Gamma_1}$.  Then $\vv{\Gamma_1}\#_3-\mir{\vv{\Gamma_2}} = \vv{\Gamma_1}$.  

We now apply Theorem \ref{thm:main}. Observe that the signature of the trivial theta vanishes,
and for any totally oriented, Klein colored theta curve we have $\beta=0$, $4\mu-12=0$, and $5-2V = 1$.  Finally, $d_{\Klein}(\vv{\Gamma_1}, \vv{\Gamma_2})$ is $u_{\Klein}(\thetan)$ since unknotting a theta curve is insensitive to total orientation and any change of Klein coloring preserves the same and mixed crossing types.  So, inequalities \eqref{ineq:a} and \eqref{ineq:c} give the desired result.  
\end{proof}

\begin{remark}\label{rem:immediatesigunknotting}
The inequality $\frac14 |\sigma(\thetan)| \leq  u_{\Klein}(\thetan)$ for $\thetan$-curves from Corollary \ref{cor:sigsliceunknot} (without the slice orbifold Euler characteristic),  can be obtained by careful application of Theorem \ref{thm:signatureforconstituentknots}, and $\frac12 |\sigma(K)| \leq  u(K)$ for a knot $K$.
\end{remark}

We begin by exploring inequalities relating signature to orbifold Euler characteristic and then establish the inequalities relating orbifold Euler characteristic to sequences of crossing changes.

\subsection{Bounding signature}\label{sec:ab}
To establish the upper bound on the signature of a Klein graph, we use the corresponding result for knots. To establish the upper bound on the difference of signatures of two Klein graphs, we use properties of the vertex sum.

Setting $z=-1$ in \cite[Theorem 1.4]{powell}, Powell shows that $|\sigma(\vv{L})| + \mu(L)-1 - \beta(\vv{L}) \leq 2g_4(\vv{L})$ so that
\begin{align*}
    |\sigma(\vv{L})| &\leq 2g_4(\vv{L}) -  \mu(L)+ 1 + \beta(\vv{L}) \\
            &= (2g_4(\vv{L}) +\mu(L)- 2) -  2\mu(L)+ 3 + \beta(\vv{L}) \\
            &= 1-\chi_4(\vv{L}) -  2(\mu(L)-1) + \beta(\vv{L}).
\end{align*}

\begin{theorem}
\label{thm:GammaSig}
For a totally oriented Klein graph $\vv{\Gamma} \subset S^3$, we have 
\[ |\sigma(\vv{\Gamma})| \leq 3-|V(\Gamma)|+2|\sv(\vv{\Gamma})|-4\chi^{orb}_4(\vv{\Gamma}) -2(\mu(\Gamma)-3) + \beta(\vv{\Gamma}).\]
\end{theorem}

\begin{proof}
By the definition of $\chi^{orb}_4$, there exists a totally orientable foamy cobordism $\vv{F}$ for $\vv{\Gamma}$ such that $\vv{F}$ has no bicolored spheres and $\chi^{orb}(\vv{F}) = \chi^{orb}_4(\vv{\Gamma})$.  

By Theorem~\ref{thm:signatureforconstituentknots}:
\begin{align*}
|\sigma(\vv{\Gamma})| &= |\sigma(\vv{\Gamma}_{rb}) + \sigma(\vv{\Gamma}_{bg}) + \sigma(\vv{\Gamma}_{rg})|\\
& \leq |\sigma(\vv{\Gamma}_{rb})| + |\sigma(\vv{\Gamma}_{bg})| + |\sigma(\vv{\Gamma}_{rg})| \\
\intertext{
Since each $\vv{\Gamma}_{ij}$ is an oriented link, and $|\sigma(\vv{L})| \leq 1-\chi_4(\vv{L})- 2(\mu(L)-1) + \beta(\vv{L})$ for any oriented link $\vv{L}$  by \cite[Theorem 1.4]{powell}, then 
}
&\leq 3-(\chi_4(\vv{\Gamma}_{rb})+\chi_4(\vv{\Gamma}_{bg})+\chi_4(\vv{\Gamma}_{rg} )) -2(\mu(\Gamma)-3) + \beta(\vv{\Gamma})
\end{align*}
Since $F$ is totally oriented with no bicolored spheres, each surface $\vv{F}_{ij}$ is oriented and has no sphere components. Then, since 
$-\chi_4(\vv{\Gamma}_{ij}) \leq -\chi(\vv{F}_{ij})$ by definition,

\begin{align*}
    -\left(\chi_4(\vv{\Gamma}_{rb})+\chi_4(\vv{\Gamma}_{bg})+\chi_4(\vv{\Gamma}_{rg})\right)
        & \leq -\left( \chi(F_{rb})+\chi(F_{bg})+\chi(F_{rg}) \right)\\
\intertext{
Since each seam edge lies on all three bicolored surfaces and each monochromatic face lies on two bicolored surfaces: 
}
        \chi(F_{rb})+\chi(F_{bg})+\chi(F_{rg}) &= (2 (\chi(F_r) + \chi(F_b) + \chi(F_g)) - 3\chi(s(F)))\\
        &= 2 (\chi(F) + 2\chi(s(F)))-3\chi(s(F))\\
        &=2\chi(F) + \chi(s(F)) \\
\intertext{By Definition~\ref{defn:chi-orb}, 
$4\chi^{orb}(F) = 2\chi(F) - \chi(s(F))$. So: }
    &= 4\chi^{orb}(F)+2\chi(s(F))\\
\intertext{
As $s(F)$ is a $1$-complex with $|V|\coloneqq|V(\Gamma)|$ boundary vertices and $|\sv|\coloneqq
|\sv(\vv{\Gamma})|$ seam vertices, it has $(|V|+4|\sv|)/2$ edges.  So $2\chi(s(F)) = 2(|V|+|sv|)-(|V|+4|\sv|) = |V|-2|\sv|$.  
 So:}
    & =  4\chi^{orb}(F)+|V|-2|\sv|\\
    &= 4\chi^{orb}_4(\vv{\Gamma})+|V|-2|\sv| 
\end{align*}
Therefore we  have
\begin{align*}
    |\sigma(\vv{\Gamma})| &\leq 3-(4\chi^{orb}_4(\vv{\Gamma})+|V|-2|\sv|)  -2(\mu(\Gamma)-3) + \beta(\vv{\Gamma})\\
    &\leq  3-|V(\Gamma)|+2|\sv|-4\chi^{orb}_4(\Gamma) -2(\mu(\Gamma)-3) + \beta(\vv{\Gamma})
\end{align*}
as claimed.
\end{proof}

\begin{lemma}\label{lem:signatureoperations}
Suppose that $\vv{\Gamma}$, $\vv{\Gamma_1}$, and $\vv{\Gamma_2}$ are totally oriented Klein graphs.  
Then:
\[ \sigma(-\vv{\Gamma}) = \sigma(\vv{\Gamma}) \qquad \mbox{and} \qquad \sigma(\mir\vv{\Gamma}) = -\sigma(\vv{\Gamma})\]

and
 \[\sigma(\vv{\Gamma_1} \#_3 \vv{\Gamma_2}) = \sigma(\vv{\Gamma_1}) + \sigma(\vv{\Gamma_2}) \]
for any vertex sum compatible with the total orientations.
\end{lemma}

\begin{proof}
    As oriented links, $(\Gamma_1 \#_3 \Gamma_2)_{ij} = (\Gamma_1)_{ij} \# (\Gamma_2)_{ij}$ for $i,j\in \{r,g,b\}$.
    Then this lemma follows from 
    Corollary~\ref{cor:signatureofTOKline}
    and the fact that signature is unchanged under orientation reversal,  negated under mirroring, and  additive for oriented links (e.g.\ \cite[Proposition 13.12]{BZ}  and  \cite[Proposition 2.13]{CimasoniFlorens}).
\end{proof}

This immediately gives:
\begin{cor}\label{cor:signaturecobordisms}
Suppose that $\vv{\Gamma_1}$ and $\vv{\Gamma_2}$ are totally oriented Klein graphs.
Then:
   \[\sigma(\vv{\Gamma_1}) - \sigma(\vv{\Gamma_2}) = \sigma(\vv{\Gamma_1} \#_3 -\mir{\vv{\Gamma_2}})\]
    where $\vv{\Gamma_1} \#_3 -\mir{\vv{\Gamma_2}}$ is any vertex sum of $\vv{\Gamma_1}$ and $-\mir{\vv{\Gamma_2}}$ compatible with the total orientations. \qed
\end{cor}

\begin{cor}
\label{two ended cobordism}
Suppose that $\vv{\Gamma_1}$ and $\vv{\Gamma_2}$ are totally oriented Klein graphs, each with at least one vertex. Let $\mu_i = \mu(\vv{\Gamma_i})$ and $V_i = |V(\vv{\Gamma_i})|$.  Then:
\[
 |\sigma(\vv{\Gamma_1}) -\sigma(\vv{\Gamma_2})| 
 \leq  3 -   (V_1 + V_2 -2)   
 +2|sv(\vv{\Gamma_1}) -\sv(\vv{{\Gamma_2}})| 
 - 4 \chi^{orb}_4(\vv{\Gamma_1} \#_3 -\mir{\vv{\Gamma_2}}) 
-2(\mu_1+\mu_2-6) 
 +\beta(\vv{\Gamma_1})+\beta( \vv{\Gamma_2}) 
\]
\noindent where $\Gamma_1 {\#_3}-\mir{\Gamma_2}$ is any vertex sum of $\vv{\Gamma_1}$ and $-\mir{\vv{\Gamma_2}}$ compatible with the total orientations.
\end{cor}

\begin{proof}
This follows from Theorem~\ref{thm:GammaSig} and Corollary~\ref{cor:signaturecobordisms} along with Lemmas~\ref{lem:invts_of_rev_mir}, \ref{lem:invts_of_sums}, and \ref{lem:sv-additive}.  
\end{proof}

\begin{lemma}\label{lem:slicecharcobordisms}
Suppose that $\vv{\Gamma}_1$ and $\vv{\Gamma}_2$ are totally oriented Klein graphs, each with at least one vertex, that cobound a totally orientable seamed foamy cobordism. Then, 
    \[
    \chi^{orb}_4(\vv{\Gamma}_1, \vv{\Gamma}_2;s) \leq \chi^{orb}_4(\vv{\Gamma}_1 \#_3 -\mir{\vv{\Gamma}_2})-\frac54
    \]
    where $\vv{\Gamma}_1 {\#_3}-\mir{\vv{\Gamma}_2}$ is  any vertex sum of $\vv{\Gamma_1}$ and $-\mir{\vv{\Gamma_2}}$ compatible with the total orientations.
\end{lemma}

\begin{proof}
    
    Let $\YGraph$  be the graph in $B^3$ that is the cone on $3$ points in $S^2$.    
    Let $F$ be a seamed foamy cobordism from $\vv{\Gamma}_1$ to $\vv{\Gamma}_2$ in $S^3 \times [0,1]$ realizing $\chi^{orb}_4(\vv{\Gamma}_1, \vv{\Gamma}_2;s)$ and having no bicolored spheres. 

    By Lemma \ref{lem:seamedfoamycobord_implies_slicefoam}, from any pair of vertices $v_1\in \vv{\Gamma}_1$ and $v_2\in\vv{\Gamma}_2$ with the same orientation, we may construct a slice foam for $\vv{\Gamma}_1\#_3-\mir\vv{\Gamma}_2$ by summing along a seam connecting $v_1$ and $v_2$.
    
    Using Definition~\ref{defn:chi-orb}, we may count 
    \[\chi^{orb}(F) 
    = \chi^{orb}(F') + \frac14 -\frac32 = \chi^{orb}(F') -\frac54.\]
    Since $F'$ may not realize $\chi^{orb}_4(\vv{\Gamma}_1 \#_3 -\mir\vv{\Gamma}_2)$, we then have the stated inequality.  
\end{proof}

\begin{remark}
    A priori, vertex sums along different pairs of vertices may produce different foamy cobordisms. 
    The slice orbifold Euler characteristic that we calculate thus depends on the choice of vertices, but for any choice the slice orbifold characteristic is bounded below by the signature difference.
\end{remark}

\subsection{Building foams from crossing changes}
\label{sec:cd}

In this section, we build foams with a view exploring the inequality \eqref{ineq:c} in Theorem \ref{thm:main}.  The two basic elementary cobordisms we employ are shown in Figure~\ref{fig:elementarycobordism}(Top Left) and (Top Center): the elementary saddle cobordism and the elementary (un)zip cobordism.

\begin{figure}
    \centering
    \includegraphics[width=.8\textwidth]{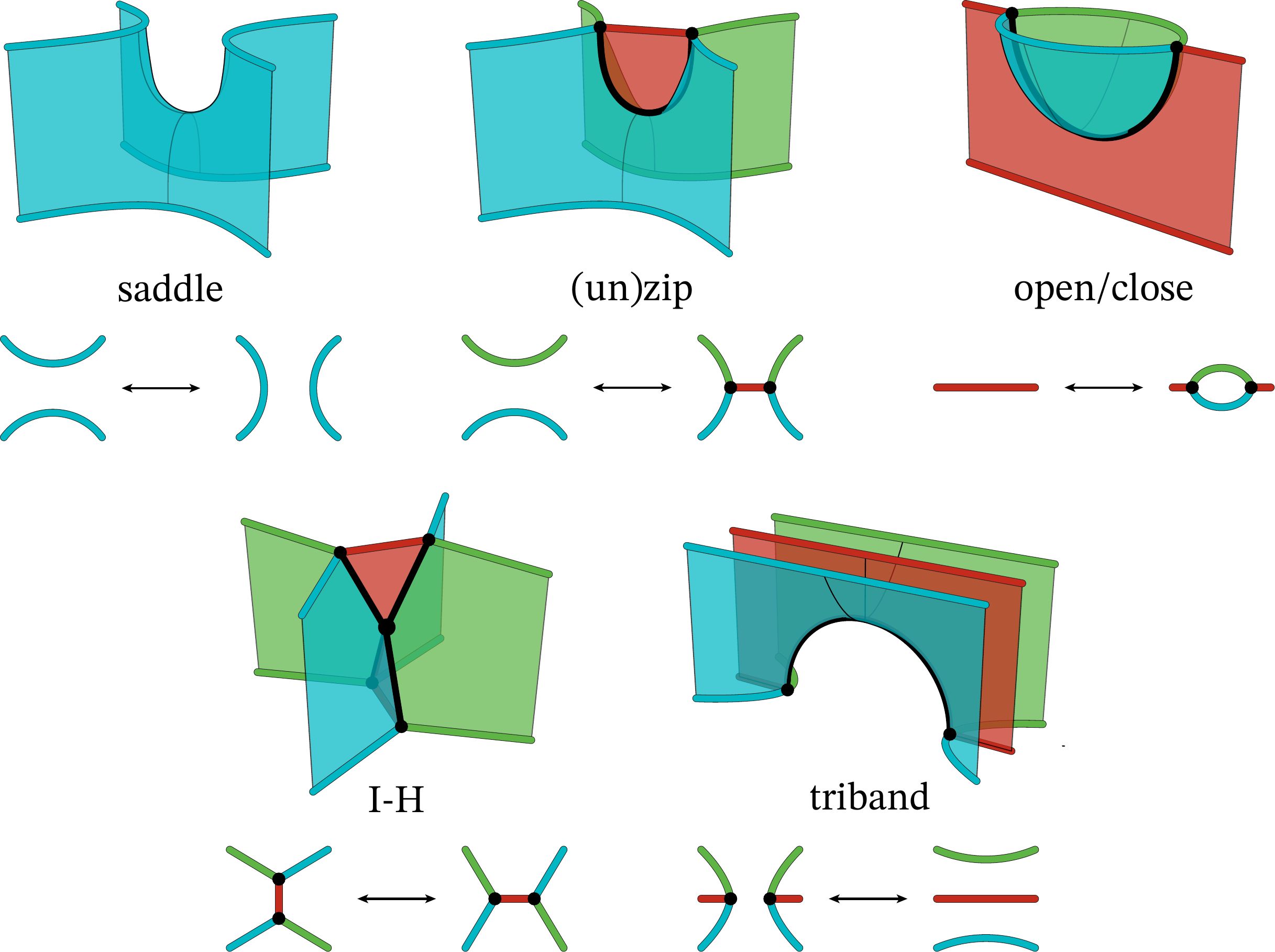}
    \caption{(Top Left) The elementary saddle cobordism.  (Top Center) The elementary (un)zip cobordism. (Top Right) The elementary open/close cobordism. (Bottom Left) The elementary I-H cobordism.  (Bottom Right) The triband cobordism.}
    \label{fig:elementarycobordism}
\end{figure}

\begin{theorem}
\label{chi-orb s-m}    
 Suppose there is a sequence of $s$ same-colored crossing changes and $m$ mixed-colored crossing changes between totally oriented Klein graphs $\vv{\Gamma_1}$ and $\vv{\Gamma_2}$.   Then there is a totally orientable seamed foamy cobordism $\vv{F}$ between $\vv{\Gamma_1}$ and $\vv{\Gamma_2}$ without bicolored spheres with 
    \[ \chi^{orb}(\vv{F}) = -|V(\vv{\Gamma_1})|/2 - (s+m/2).\]
\end{theorem}

\begin{remark}\label{rem:extending_TOs_to_foams}
    Crossing changes do not alter total orientations on the edges involved. 
    Indeed, any choice of total orientation on $\Gamma_1$ confers a total orientation on $\Gamma_2$ and extends to a total orientation on the seamed foamy cobordism $F$ constructed in the proof of Theorem~\ref{chi-orb s-m}.
\end{remark}

\begin{figure}[!htbp]
    \centering
    \includegraphics[width=.9\textwidth]{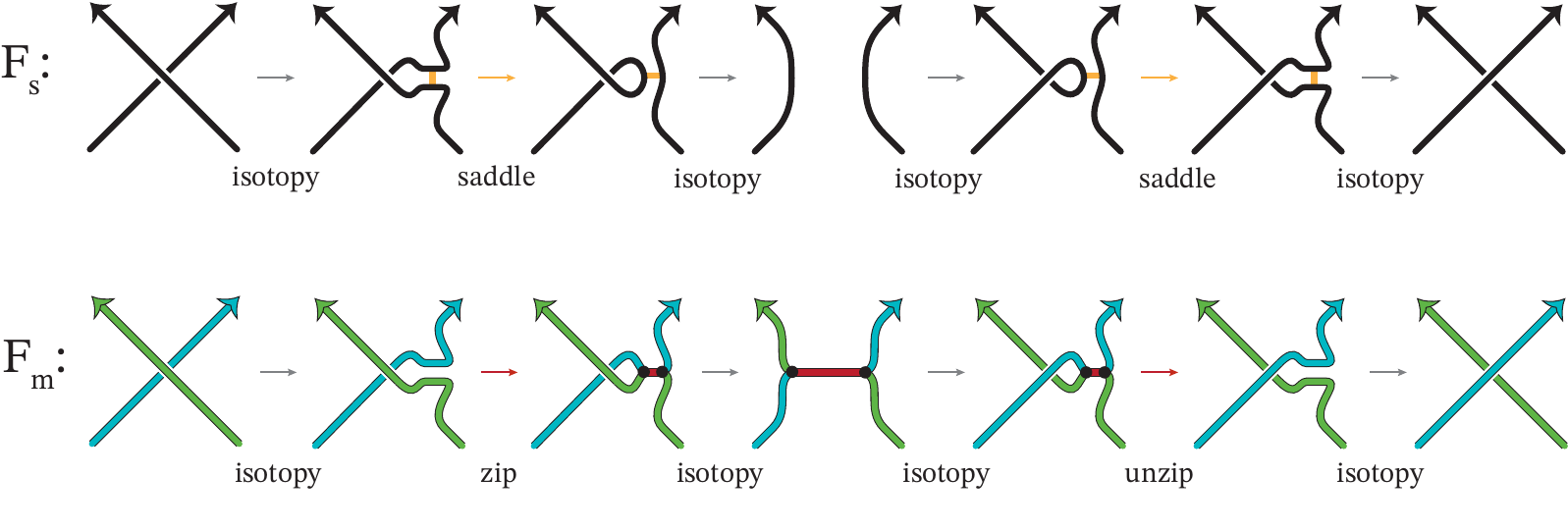}
    \caption{Top: Movie of an oriented same crossing change cobordism $F_s$ realized by two elementary saddle cobordisms.  Bottom: Movie of an oriented mixed crossing change  cobordism $F_m$  realized by a pair of elementary zip/unzip cobordisms.}
    \label{fig:crossingchangecobordism}
\end{figure}

\begin{figure}
    \centering
    \includegraphics[width=.8\textwidth]{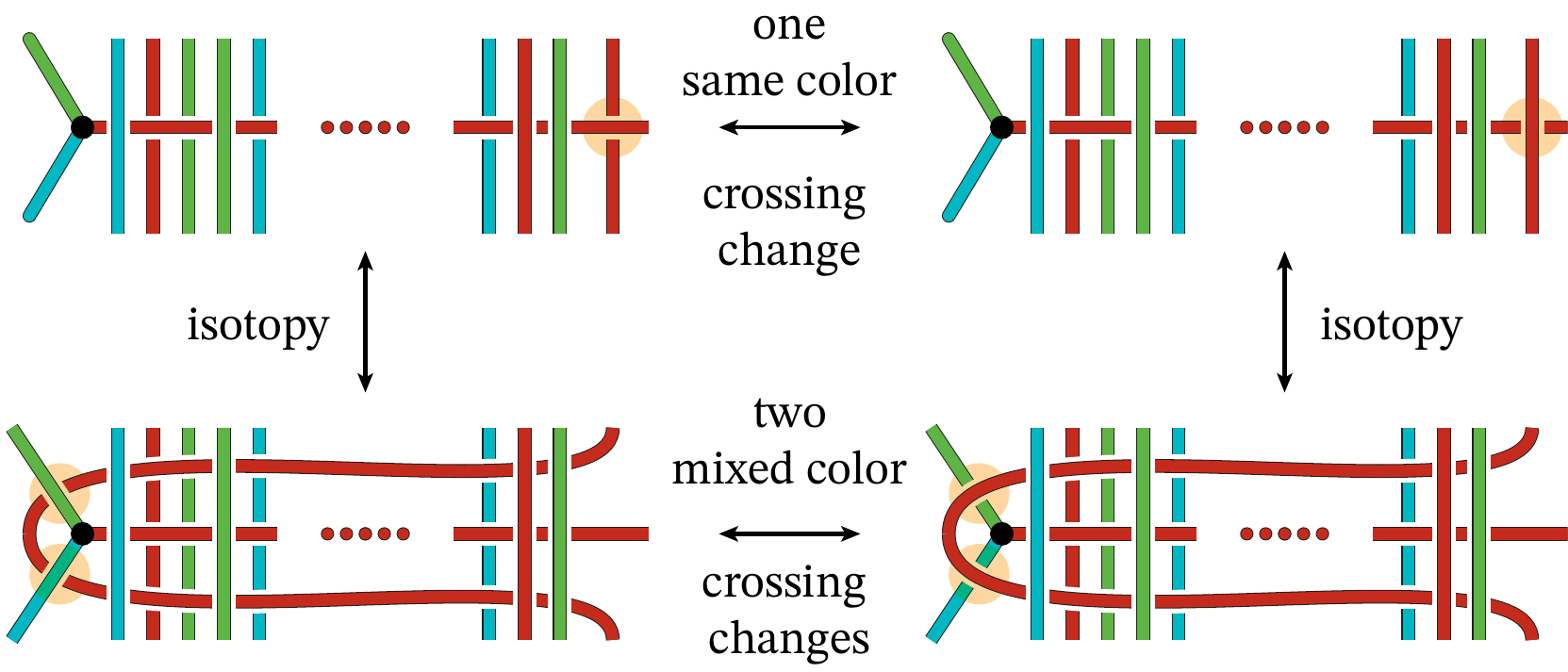}
    \caption{A same crossing change may be realized as two mixed crossing changes.}
    \label{fig:sametomixed}
\end{figure}

\begin{proof}[Proof of Theorem \ref{chi-orb s-m}]
In the appendix of \cite{GR}, a foam is constructed from a sequence of crossing changes on a Klein graph.  However, the constructed foam is not necessarily totally orientable.   Here we review such a construction and how to alter it in order to produce a totally orientable foamy cobordism. The key step is to replace any problematic same-colored crossing change by a pair of mixed-colored crossing changes. 
Thereafter, we determine the Euler characteristic of the constructed foam.

Mirroring as needed, Figure~\ref{fig:crossingchangecobordism} shows, via a movie, how single same-colored or mixed-colored crossing changes of oriented arcs of a Klein graph can be realized as an ``elementary'' foamy cobordism of type $F_s$ or $F_m$ from $\Gamma$ to $\Gamma_s$ or $\Gamma_m$, respectively. Observe that $F_s$ is assembled using two elementary saddle cobordisms while $F_m$ is assembled using a pair of elementary zip/unzip cobordisms. See Figure~\ref{fig:elementarycobordism}.  Further note that both $F_s$ and $F_m$ are seamed foamy cobordisms.

Stacking these cobordisms together according to a crossing change sequence from a totally oriented Klein graph $\vv{\Gamma_1}$  to a totally oriented Klein graph $\vv{\Gamma_2}$ then assembles a seamed foamy cobordism $F$ between the two graphs.  However, one must be more careful to ensure the foamy cobordism $F$ admits a total orientation that respects the total orientations of the two graphs.

Observe that crossing changes do not alter the double orientations on edges of a totally oriented Klein graph.  
Hence, through the sequence of crossing changes, the total orientation on $\Gamma_1$ induces a total orientation on each of the intermediate graphs and is carried to the total orientation on $\Gamma_2$.
To make the elementary crossing change foams piece together into a totally oriented foam, the elementary crossing change cobordisms need to respect these total orientations.

\medskip
Figure~\ref{fig:crossingchangecobordism} (Top) shows how a saddle-saddle sequence produces a foamy cobordism associated to a same-color crossing change of oriented arcs. This same-color crossing change cobordism $F_s$ between Klein graphs $\Gamma$ and $\Gamma_s$ is built from two saddle moves that respect a choice of orientations on the arcs involved. Together these saddles form a 2-dimensional 1-handle that is attached in a manner that respects the orientation chosen on these arcs of $\Gamma$.  (The slice genus bound on unknotting number for knots follows from stacking such cobordisms.) 

Note that a same-colored crossing change involves either one or two edges.
When the double orientations of the edges involved in a same-colored crossing change have the same sign (are all parallel or all antiparallel), one may observe that $F_s$ is totally orientable in a manner that respects the total orientations of $\Gamma$ and the resulting $\Gamma_s$.
However, suppose a same-colored crossing change involved two distinct $i$-colored edges where the edges have opposite sign.  Then one of the surfaces $(F_s)_{ij}$ and $(F_s)_{ik}$ is orientable while the other is non-orientable.  Hence, the cobordism $F_s$ associated to a same-color crossing change is totally orientable (with a total orientation respecting the total orientations of $\Gamma$ and $\Gamma_s$) exactly when the edges involved have the same sign of double orientation.  Fortunately, as shown in Figure~\ref{fig:sametomixed},  a problematic same-color crossing change (where the edges have different signs) can be realized as a pair of mixed-color crossing changes.

\medskip

Figure~\ref{fig:crossingchangecobordism} (Bottom)  shows how a zip-unzip sequence produces a seamed foamy cobordism associated to a mixed-color crossing change of oriented arcs.   This mixed-color crossing change cobordism $\vv{F_m}$ between totally oriented Klein graphs $\vv{\Gamma}$ and $\vv{\Gamma_m}$ is homeomorphic to the totally oriented trivial cobordism $\vv{\Gamma} \times I$ with a $1$-handle joining the facet surfaces of the two colors of the crossing change along with a disk of the third color that is the co-core of the handle. 
(This move adds a closed component to the seams.) 
Suppose the mixed-color crossing changes involves an $i$-colored and a $j$-colored edge.  When forming $\vv{F_m}$, the  $1$-handle  is attached to $\vv{\Gamma} \times I$ so that $(\vv{F_m})_{ij}$ is oriented, extending the orientation of $(\vv{\Gamma} \times I)_{ij}$.   Then, since $(\vv{F_m})_{ik}$ and  $(\vv{F_m})_{jk}$ are identified with $(\vv{\Gamma \times I})_{ik}$ and $(\vv{\Gamma \times I})_{jk}$ by having an interior disk of $(\vv{\Gamma} \times I)_{i}$ and $(\vv{\Gamma} \times I)_{j}$ changed to a $k$-colored disk, the surfaces are oriented.  Indeed, this shows that the bicolored surfaces of $\vv{F_m}$ are oriented in a manner that respect the total orientations of $\Gamma$ and $\Gamma_m$. Hence $\vv{F_m}$ is a totally oriented cobordism between  $\vv{\Gamma}$ and $\vv{\Gamma_m}$.

Now, given a sequence of crossing changes from $\vv{\Gamma_1}$ to $\vv{\Gamma_2}$, exchange any same-colored crossing change between a parallel edge and an antiparallel edge for a pair of mixed-colored crossings changes so that the crossing changes are realized by totally oriented elementary foamy cobordisms.
Then stack all these totally oriented elementary foamy cobordisms together to assemble a totally oriented foamy cobordism $\vv{F}$ from $\vv{\Gamma_1}$ to $\vv{\Gamma_2}$.   

Furthermore, one readily observes that each component of each bicolored surface of a mixed-color crossing change cobordism $\vv{F_m}$ between $\vv{\Gamma}$ and $\vv{\Gamma_m}$ has boundary meeting each $\vv{\Gamma}$ and $\vv{\Gamma_m}$.  Indeed, the only component of $(\Gamma_m)_{ij}$ which is not identified with a component of $(\vv{\Gamma} \times I)_{ij}$ is the one obtained by a 1-handle attachment.  Similarly, each component of each bicolored surface of a totally oriented same-colored crossing change cobordism $\vv{F_s}$ meets both $\vv{\Gamma}$ and $\vv{\Gamma_s}$.
Consequently, a totally oriented foamy cobordism  from $\vv{\Gamma_1}$ to $\vv{\Gamma_2}$ obtained by stacking totally oriented elementary crossing change cobordisms has no bicolored spheres. 

\medskip

Now we calculate the Euler characteristics of these elementary cobordisms.   (Total orientations play no role here.)
According to the above descriptions, the orbifold Euler characteristics of these elementary crossing change cobordisms are calculated to be 
\[\chi^{orb}(F_s) = \chi^{orb}(\Gamma \times I) - \frac12 \cdot 2 = \chi^{orb}(\Gamma \times I) - 1\] 
and 
\[\chi^{orb}(F_m) = \chi^{orb}(\Gamma \times I) - \frac12 \cdot 1 = \chi^{orb}(\Gamma \times I) - \frac12.\] 
So now given a sequence of $s$ same-colored crossing changes and $m$ mixed-colored crossing changes between Klein graphs $\Gamma_1$ and $\Gamma_2$ where the edges of $\Gamma_1$ have been assigned some orientations, let $F$ be the foamy cobordism from $\Gamma_1$ to $\Gamma_2$ obtained by stacking the associated elementary crossing change cobordisms.  Hence
\[\chi^{orb}(F) = \chi^{orb}(\Gamma_1 \times I) - s - m/2.\] 
Finally, the seam edges join the vertices of the components of $\Gamma_1 \times \boundary I$, so $\chi^{orb}(\Gamma_1 \times I) = \chi(\Gamma_1)$. Hence, $\chi^{orb}(F) = \chi(\Gamma_1) - s - m/2 = -\frac12 |V(\Gamma_1)| - s -m/2$, as claimed.

Observe that the contribution to $\chi^{orb}(F)$ of an elementary same-color crossing change cobordism is twice that of a mixed-color crossing change cobordism.  Since each same-colored crossing change may be realized by two mixed-color crossing changes, we may instead build a totally oriented foamy cobordism $\vv{F}'$ from $\vv{\Gamma_1}$ to $\vv{\Gamma_2}$ by stacking $m+2s$ elementary mixed-color crossing change cobordisms so that $\chi^{orb}(\vv{F}') = \chi^{orb}(\vv{F})$.  Furthermore, as discussed above, such a totally oriented foamy cobordism $\vv{F}'$ will have no bicolored spheres.
\end{proof}

\begin{cor}
\label{chi-orb dist}
Suppose the Klein Gordian distance between two Klein graphs $\Gamma_1$ and $\Gamma_2$ is $d_{\Klein}(\Gamma_1, \Gamma_2)$. Then 
        \[ -\chi_4^{orb}(\Gamma_1, \Gamma_2;s) \leq   d_{\Klein}(\Gamma_1, \Gamma_2)+\frac12 |V(\Gamma_1)|. 
        \]
\end{cor}

\begin{proof}
 Say $d_{\Klein}(\Gamma_1, \Gamma_2)$ is realized by a sequence of $s$ same-colored crossing changes and $m$ mixed-colored crossing changes.  Then  by Theorem~\ref{chi-orb s-m} there is a seamed foamy cobordism $F$ such that 
 $\chi^{orb}(F) = -\frac12 V(\Gamma_1) - d_{\Klein}(\Gamma_1, \Gamma_2)$.
 Since  $\chi^{orb}(F) \leq \chi_4^{orb}(\Gamma_1, \Gamma_2;s)$ by Definition \ref{defn:chi-orb}, 
 we have 
 $-\frac12 V(\Gamma_1) - d_{\Klein}(\Gamma_1, \Gamma_2) \leq \chi_4^{orb}(\Gamma_1, \Gamma_2;s)$. 
 The inequality follows.
\end{proof}

\section{Examples}
\label{sec:examples}

We present several examples of how our bounds shed light on the unknotting number of $\theta$-curves; our bounds improve previously known bounds. We conclude by presenting examples of $\theta$-curves showing that the stronger form of Gille-Robert's signature does not produce bounds on unknotting number.

\subsection{Improved bounds}
\label{subsec:improvebounds}
Let the standard unknotting number $u(\thetan)$ be the minimum  number of crossing changes needed to unknot $\thetan$.

This is known to be bounded below by the maximal constituent unknotting number of the graph, that is $mcu(g)=max\{u(s)|\text{ where } s \text{ is a constituent of } g\}$ \cite{BOD}.   (For any Klein coloring of a $\thetan$-curve, a constituent knot is a bicolored knot.) The following is a family of $\thetan$-curves where our new bound $\frac{1}{4}|\sigma(\thetan)|\leq u_{\Klein}(\thetan)$ significantly improves on the past result.

\begin{figure}[!htbp]
    \centering
    \includegraphics[width=.8\textwidth]{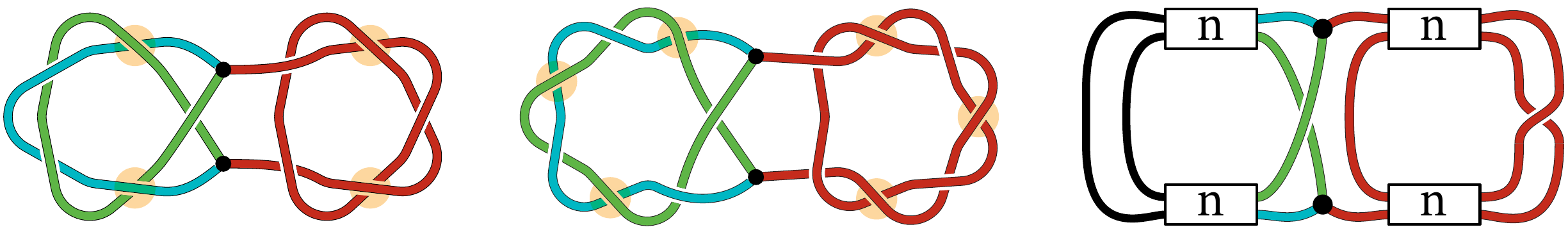} 
    \caption{(Left) A $\thetan$-curve $\thetan_2$ for which every constituent knot is the $(2,5)$-torus knot $5_1$. Highlighted are a set of $4$ crossing changes that unknot $\thetan_2$. (Center) A $\thetan$-curve $\thetan_3$ for which every constituent knot  is the $(2,7)$-torus knot $7_1$.  Highlighted are a set of $6$ crossing changes that unknot $\thetan_3$.  (Right) A generalization to a family of $\thetan$-curves  $\thetan_n$ for which every constituent knot is the $(2,2n+1)$-torus knot.}
    \label{fig:unknottingexamples}
\end{figure}

In the theta curve $\thetan_2$ of Figure \ref{fig:unknottingexamples}(Left), all of the constituent knots are $5_1$, so $mcu(\thetan_2)=u(5_1)=2$. (Note $|\sigma(5_1)|=4$.)  Thus the earlier bounds gives $2\leq u(\thetan_2)$.  Next $|\sigma(\thetan_2)| = 3|\sigma(5_1)|=3\cdot4=12.$  So our new bound gives $\frac{1}{4}|\sigma(\thetan_2)|=3 \leq u_{\Klein}(\thetan_2)$.  In the figure there is a set of 2 same crossings and 2 mixed crossing that will unknot $\thetan_2$, giving $s+\frac{m}{2}=3$.
Hence $u_{\Klein}(\thetan_2)=3$ and our lower bound on $u_{\Klein}$ is sharp in this case.  Furthermore, $u_{\Klein}(\thetan_2)=3$ implies that $u(\thetan_2) \geq 3$, improving the lower bound on unknotting number from $2$ to $3$.

Similarly, in the theta curve $\thetan_3$ of Figure \ref{fig:unknottingexamples}(Center), all of the constituent knots are $7_1$, so $mcu(\thetan_3)=u(7_1)=3$. (Note $|\sigma(7_1)|=6$.)  Thus the earlier bounds gives $3\leq u(\thetan_3)$.  Next $|\sigma(\thetan_3)| = 3|\sigma(7_1)|=3\cdot 6=18.$  So our new bound gives $\frac{1}{4}|\sigma(\thetan_3)|=4.5 \leq u_{\Klein}(\thetan_3)$.  In the figure there is a set of 3 same crossings and 3 mixed crossing that will unknot $\thetan_3$, giving $s+\frac{m}{2}=4.5$.
Hence $u_{\Klein}(\thetan_3)=4.5$ and our lower bound on $u_{\Klein}$ is sharp in this case.  Furthermore, $u_{\Klein}(\thetan_3)=4.5$ implies that $u(\thetan_3) \geq 5$, improving the lower bound on unknotting number from $3$ to $5$.  

Generalizing these examples further gives a family of $\thetan$-curves $\thetan_n$ for $n \in \N$ where all of the constituent knots are torus knots $T(2,2n+1)$, shown in Figure \ref{fig:unknottingexamples}(Right).  Here $mcu(\thetan_n)=u(T(2,2n+1))=n$.  Thus the earlier bound gives $n\leq u(\thetan_n)$ and our new bound gives $\frac{1}{4}|\sigma(\thetan_n)|=\frac32 n \leq u_{\Klein}(\thetan_n)$. Similarly these examples have a set of $n$ same crossings and $n$ mixed crossing that will unknot $\thetan_n$, giving $s+\frac{m}{2}=\frac32 n$.  Hence   $u_{\Klein}(\thetan_n)=\frac32 n$ while $u(\thetan_n) \leq 2n$.

\begin{question}\label{ques:unknottingexample}
While $u(\thetan_n) \leq 2n$,
our count $u_{\Klein}(\thetan_n)=\frac32 n$ implies that $u(\thetan_n) \geq \lceil \frac32 n \rceil$. 
Is $u(\thetan_n) < 2n$ for some $n \geq 2$?
\end{question}

\begin{remark}
    Observe that $u(\thetan_2) \in \{3,4\}$.  If $u(\thetan_2)=3$, then $\thetan_2$ would be unknotted by a set of $3$ same crossing changes.

    Similarly, $u(\thetan_3) \in \{5,6\}$.  If $u(\thetan_3)=5$, then $\thetan_3$ would be unknotted by a set of $4$ same crossing changes and $1$ mixed crossing change.
\end{remark}

\begin{remark}
    Each $\thetan_n$ is the edge connected sum of a theta curve $\thetan_n'$ and the knot $T(2,2n+1)$.  One observes that $mcu(\thetan_n')=u(T(2,2n+1))=n$ so both of these have unknotting number $n$.
    Hence if $u(\thetan_n) < 2n$ for some $n \geq 2$, answering Question~\ref{ques:unknottingexample} in the affirmative, then the unknotting number for theta curves would be not additive under edge connected sum. 
\end{remark}

\subsection{Strong signature}
\label{subsec:strong sig}

In their article, Gille-Robert define two signature invariants, $\sigma$ and $\tilde{\sigma}$ for 3-Hamiltonian Klein graphs; we term them the \emph{weak signature} and the \emph{strong signature}. The signature used through this article (which we extend to all Klein graphs) is the weak signature $\sigma(\Gamma)$, defined  as
\[\sigma(\Gamma) = \sigma(W_F) + \frac12e(F),\]
and is equivalent to the sum of the signatures of the constituent knots $\sigma(\Gamma)=\sigma(K_{rb})+\sigma(K_{bg})+\sigma(K_{rg})$ \cite[Proposition 3.14]{GR}. 
The term $\sigma(W_F)$ is the signature of the 4-manifold Klein cover of the 4-ball branched over a Klein foam $F$ for $\Gamma$, and $e(F)$ is the \emph{weak} Euler number of the foam $F$, given by the sum of relative normal Euler numbers of the facets, $e(F_{rb})+e(F_{bg})+e(F_{rg})$. The \emph{strong} signature is defined as
\[\tilde{\sigma}(\Gamma) = \sigma(W_F) + \frac12\tilde{e}(F),\]
where $\tilde{e}(F)$ is the \emph{strong} Euler number \cite[Definition 3.10]{GR}. The strong Euler number is given by the sum of relative normal Euler numbers of the surfaces $\tilde{F}^r_r + \tilde{F}^b_b +\tilde{F}^g_g$. Each surface $\tilde{F}^i_i$, $i=r, g, b$ is the image of the fixed point surface of the diffeomorphism induced by the action on $W_F$ of the element $i$ in the Klein group. It is shown in \cite[Proposition 3.14]{GR} that $\tilde{\sigma}(\Gamma)$ can be calculated from a sum $\tilde{\Gamma} = \xi(\tilde{\Gamma}_r^r) + \xi(\tilde{\Gamma}_b^b) +  \xi(\tilde{\Gamma}_g^g)$ of signatures of the links $\tilde{\Gamma}_i^i$ in the rational homology spheres bounding the 4-manifolds $W/i$ (which we can also think of as double covers of the 4-ball branched over the surfaces $F_{jk}$). See \cite[Defn 3.10]{GR}, \cite[Thm 3.11]{GR}, \cite[Definition 3.12]{GR}.  

It is natural to ask whether the bounds of Theorem \ref{thm:main} extend to strong signature. In the following example, we use the Kinoshita-Wolcott family of theta-curves  \cite{Kinoshita-AlexanderI, Wolcott} to show that, unlike the weak signature,  strong signature does not yield a successful bound.

\begin{figure}[h]
    \centering
    \includegraphics[width=.4\textwidth]{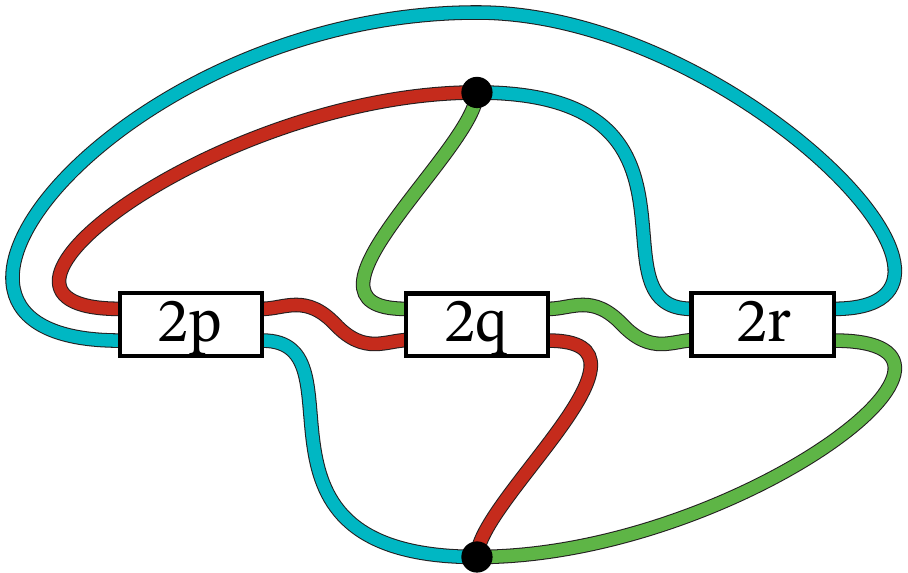}
    \caption{The Kinoshita-Wolcott $\thetan$-curve $\theta(p,q,r)$.  \cite{Kinoshita-AlexanderI, Wolcott}}
    \label{fig:kinoshita}
\end{figure}

Let $\theta(p,q,r)$ be the Kinoshita-Wolcott $\theta$-curve with $p$, $q$ and $r$ full twists shown in Figure~\ref{fig:kinoshita}. Notice that all constituent knots are unknotted, thus all edges lift to knots in the three-sphere, which we may use for the calculation of strong signature. Then:
\[\widetilde{\sigma}(\theta(p,q,r)) = \begin{cases}
  -8(p+q+r)  &\text{if $p$, $q$, and $r$ are odd} \\
  -8p &\text{if $p$ is odd and $q$ and $r$ are even}\\
  -8q &\text{if $q$ is odd and $p$ and $r$ are even}\\
  -8r &\text{if $r$ is odd and $p$ and $q$ are even}\\
  0 & \text{else}
\end{cases}
\]

So if we take $p=1$ and $q$ and $r$ positive odd integers, we get an infinite family of examples with $u(\theta(p,q,r))=1$, and $|\widetilde{\sigma}(\theta(p,q,r))|=8(1+q+r)>1$.

\section{Appendix}

Here we aggregate information about the behavior of linking number, nullity, component count, signature, and seam vertex count under reversal, mirroring and connected sum.

\begin{table}[h]
\begin{tabular}{lllll}
\toprule
 invariant &    mirror &  reverse  \\
 \midrule
 $\lambda$ &$\lambda(\mir(\vv{\Gamma}))=-\lambda(\vv{\Gamma})$ &$\lambda(-\vv{\Gamma})=\lambda(\vv{\Gamma})$  \\ 
 
 $\beta$ & $\beta(\mir(\vv{\Gamma}))=\beta(\vv{\Gamma})$ & $\beta(-\vv{\Gamma})=\beta(\vv{\Gamma})$ \\
 
 $\mu$  & $\mu(\mir(\vv{\Gamma}))=\mu(\vv{\Gamma})$& $\mu(-\vv{\Gamma})=\mu(\vv{\Gamma})$ \\

 $\sigma$ & $\sigma(\mir\vv{\Gamma}) = -\sigma(\vv{\Gamma})$ & $\sigma(-\vv{\Gamma}) = \sigma(\vv{\Gamma})$ \\
 
 $\sv$ & $\sv(\mir{\vv{\Gamma}}) = \sv(\vv{\Gamma})$ & $\sv(-\vv{\Gamma}) = -\sv(\vv{\Gamma})$ \\
\bottomrule\\
\end{tabular}
\caption{The behavior of linking number, nullity, component count, signature, and signed seam vertex count under reversal and mirroring, from Lemma \ref{lem:invts_of_rev_mir}, Definition \ref{defn:signedsvcount}, and Lemma \ref{lem:signatureoperations}.
}
\end{table}

\begin{table}[h]
\begin{tabular}{lll}
\toprule
 invariant  & $\#_2$ &  $\#_3$ \\
 \midrule
 $\beta$  & $\beta({\Gamma_1} \#_2 {\Gamma_2}) =  \beta({\Gamma_1}) + \beta({\Gamma_2})+1$ & 
 $\beta({\Gamma_1} \#_3 {\Gamma_2}) =  \beta({\Gamma_1}) + \beta({\Gamma_2}) $ \\
 
 $\mu$  & $\mu({\Gamma_1} \#_2 {\Gamma_2}) =  \mu({\Gamma_1}) + \mu({\Gamma_2})-2$ & $\mu({\Gamma_1} \#_3 {\Gamma_2}) =  \mu({\Gamma_1}) + \mu({\Gamma_2})-3$   \\

 $\sigma$  && $\sigma(\vv{\Gamma_1} \#_3 \vv{\Gamma_2}) = \sigma(\vv{\Gamma_1}) + \sigma(\vv{\Gamma_2})$ \\

 $\sv$ &$\sv(\vv{\Gamma_1} \#_2 \vv{\Gamma_2}) = \sv(\vv{\Gamma_1}) + \sv(\vv{\Gamma_2})$ & $\sv(\vv{\Gamma_1} \#_3 \vv{\Gamma_2}) = \sv(\vv{\Gamma_1}) + \sv(\vv{\Gamma_2})$ \\

\bottomrule\\
\end{tabular}
\caption{The behavior of nullity, component count, signature, and signed seam vertex count under connected sum from Lemmas \ref{lem:invts_of_sums}, Lemma \ref{lem:sv-additive}, and Lemma \ref{lem:signatureoperations}. Signature under order two connected sum is dependent on edge choice and is omitted.}
\end{table}

\section*{Acknowledgements}
 
This paper was initiated as part of the SQuaRE program at the American Institute of Mathematics (AIM). We thank AIM for their support and hospitality.
KLB was partially supported by the Simons Foundation gift \#962034. He also thanks the University of Pisa for their hospitality where he was a visitor while this work was completed.
AHM is partially supported by NSF Grant DMS--2204148 and The Thomas F. and Kate Miller Jeffress Memorial Trust, Bank
of America, Trustee. 
SAT is partially supported by NSF Grant DMS-2104022 and a Colby College Research Grant.

\bibliographystyle{alpha}
\bibliography{biblio}

\newcommand{\etalchar}[1]{$^{#1}$}
\begin{thebibliography}{BBM{\etalchar{+}}22}

\bibitem[BBM{\etalchar{+}}22]{BBMOT}
Kenneth~L. Baker, Dorothy Buck, Allison~H. Moore, Danielle O'Donnol, and Scott
  Taylor.
\newblock Primality of theta-curves with proper rational tangle unknotting
  number one.
\newblock arXiv:2201.08213 [math.GT], 2022.

\bibitem[BO18]{BOD}
Dorothy Buck and Danielle O'Donnol.
\newblock Unknotting numbers for prime $\theta$-curves up to seven crossings.
\newblock arXiv:1710.05237v2 [math.GT], 2018.

\bibitem[BZH14]{BZ}
Gerhard Burde, Heiner Zieschang, and Michael Heusener.
\newblock {\em Knots}, volume~5 of {\em De Gruyter Studies in Mathematics}.
\newblock De Gruyter, Berlin, extended edition, 2014.

\bibitem[CF08]{CimasoniFlorens}
David Cimasoni and Vincent Florens.
\newblock Generalized {S}eifert surfaces and signatures of colored links.
\newblock {\em Trans. Amer. Math. Soc.}, 360(3):1223--1264, 2008.

\bibitem[CHK00]{Cooper}
Daryl Cooper, Craig~D. Hodgson, and Steven~P. Kerckhoff.
\newblock {\em Three-dimensional orbifolds and cone-manifolds}, volume~5 of
  {\em MSJ Memoirs}.
\newblock Mathematical Society of Japan, Tokyo, 2000.
\newblock With a postface by Sadayoshi Kojima.

\bibitem[GL78]{GordonLitherland}
C.~McA. Gordon and R.~A. Litherland.
\newblock On the signature of a link.
\newblock {\em Invent. Math.}, 47(1):53--69, 1978.

\bibitem[GR18]{GR}
Catherine Gille and Louis-Hadrien Robert.
\newblock A signature invariant for knotted {K}lein graphs.
\newblock {\em Algebr. Geom. Topol.}, 18(6):3719--3747, 2018.

\bibitem[Kin58]{Kinoshita-AlexanderI}
Shin'ichi Kinoshita.
\newblock Alexander polynomials as isotopy invariants. {I}.
\newblock {\em Osaka Math. J.}, 10:263--271, 1958.

\bibitem[KR21]{KR}
Mikhail Khovanov and Louis-Hadrien Robert.
\newblock Foam evaluation and {K}ronheimer-{M}rowka theories.
\newblock {\em Adv. Math.}, 376:Paper No. 107433, 59, 2021.

\bibitem[KT76]{KauffmanTaylor}
Louis~H. Kauffman and Laurence~R. Taylor.
\newblock Signature of links.
\newblock {\em Trans. Amer. Math. Soc.}, 216:351--365, 1976.

\bibitem[Mas69]{Mason-homeocurves}
W.~K. Mason.
\newblock Homeomorphic continuous curves in {$2$}-space are isotopic in
  {$3$}-space.
\newblock {\em Trans. Amer. Math. Soc.}, 142:269--290, 1969.

\bibitem[Mur65]{Murasugi}
Kunio Murasugi.
\newblock On a certain numerical invariant of link types.
\newblock {\em Trans. Amer. Math. Soc.}, 117:387--422, 1965.

\bibitem[OSB18]{BOS}
Danielle O'Donnol, Andrzej Stasiak, and Dorothy Buck.
\newblock Two convergent pathways of {DNA} knotting in replicating dna
  molecules as revealed by $\theta$-curve analysis.
\newblock {\em Nucleic Acids Res}, 46(17):9181--9188, Sep 2018.

\bibitem[Pow17]{powell}
Mark Powell.
\newblock The four-genus of a link, levine–tristram signatures and
  satellites.
\newblock {\em Journal of Knot Theory and Its Ramifications}, 26(02):1740008,
  2017.

\bibitem[Wol87]{Wolcott}
Keith Wolcott.
\newblock The knotting of theta curves and other graphs in {$S^3$}.
\newblock In {\em Geometry and topology ({A}thens, {G}a., 1985)}, volume 105 of
  {\em Lecture Notes in Pure and Appl. Math.}, pages 325--346. Dekker, New
  York, 1987.

\end{thebibliography}

\end{document}